\documentclass[sn-mathphys,Numbered]{sn-jnl}

\usepackage{graphicx}%
\usepackage{multirow}%
\usepackage{amsmath,amssymb,amsfonts}%
\usepackage{amsthm}%
\usepackage{mathrsfs}%
\usepackage[title]{appendix}%
\usepackage{xcolor}%
\usepackage{textcomp}%
\usepackage{manyfoot}%
\usepackage{booktabs}%
\usepackage{algorithm}%
\floatname{algorithm}{Algorithm}
\usepackage{algorithmic,float}
\usepackage{listings}%

\theoremstyle{thmstyleone}%
\newtheorem{theorem}{Theorem}
\theoremstyle{thmstyletwo}%
\newtheorem{lemma}{Lemma}%
\theoremstyle{thmstylethree}%
\newtheorem{corollary}{Corollary}
\newtheorem{assume}{Assume}
\usepackage{footmisc}   
\usepackage[misc]{ifsym}

\usepackage[numbers,sort&compress]{natbib}
\begin{document}

\title[Article Title]{Optimization methods for solving matrix equations}


\author*[1]{\fnm{Juan} \sur{Zhang}}
\email{zhangjuan@xtu.edu.cn}
\author*[2]{\fnm{Xiao} \sur{Luo}}

\affil*[1]{Key Laboratory of Intelligent Computing and Information Processing
of Ministry of Education, Hunan Key Laboratory for Computation and Simulation in Science and
Engineering, School of Mathematics and Computational Science, Xiangtan University, Xiangtan,
Hunan, China}

\affil*[2]{School of Mathematics and Computational Science, Xiangtan University, Xiangtan, Hunan,
China}
%
%


\abstract{In this paper, we focus on using optimization methods to solve matrix equations by transforming the problem of solving the Sylvester matrix equation or continuous algebraic Riccati equation into an optimization problem. Initially, we use a constrained convex optimization method (CCOM) to solve the Sylvester matrix equation with $\ell_{2,1}$-norm, where we provide a convergence analysis and numerical examples of CCOM; however, the results show that the algorithm is not efficient. To address this issue, we employ classical quasi-Newton methods such as DFP and BFGS algorithms to solve the Sylvester matrix equation and present the convergence and numerical results of the algorithm. Additionally, we compare these algorithms with the CG algorithm and AR algorithm, and our results demonstrate that the presented algorithms are effective. Furthermore, we propose a unified framework of the alternating direction multiplier method (ADMM) for directly solving the continuous algebraic Riccati equation (CARE), and we provide the convergence and numerical results of ADMM. Our experimental results indicate that ADMM is an effective optimization algorithm for solving CARE. Finally, to improve the effectiveness of the optimization method for solving Riccati equation, we propose the Newton-ADMM algorithm framework, where the outer iteration of this method is the classical Newton method, and the inner iteration involves using ADMM to solve Lyapunov matrix equations inexactly. We also provide the convergence and numerical results of this algorithm, which our results demonstrate are more efficient than ADMM for solving CARE.}

\keywords{Sylvester matrix equation, $\ell_{2,1}$-norm, quasi-Newton methods, CARE, Newton's method, ADMM, optimization problem}



\maketitle

\section{Introduction}\label{sec1}

The problem of solving matrix equations is a significant area of interest in the fields of cybernetics and matrix computation. It has wide-ranging applications in various domains of applied mathematics and holds great importance. As a result, extensive research has been conducted by scholars on different types of matrix equations \cite{ref1}-\cite{ref7}. The numerical solution of the Sylvester matrix equation is particularly relevant in the fields of cybernetics, signal processing, neural networks, model reduction, and image restoration.

   To begin, we focus on the Sylvester matrix equation, which is expressed as follows:
   \begin{equation}\label{eq0}
   AX + XB = C,
   \end{equation}
   where $A \in \mathbb{R}^{m \times m},B \in \mathbb{R}^{n \times n},C \in \mathbb{R}^{m \times n}$ are given matrices, and $X \in \mathbb{R}^{m \times n}$ is an unknown matrix to be solved.

   In recent times, there has been significant interest in finding solutions and performing numerical calculations for the Sylvester matrix equation. The Bartels-Stewart method \cite{ref30} and the Hessenberg-Schur method \cite{ref31} are commonly employed techniques for solving this equation, particularly when dealing with small and dense system matrices. For smaller system matrices, block Krylov subspace methods \cite{ref32,ref33} and global Krylov subspace methods \cite{ref34} have been proposed. These methods leverage the global Arnoldi process, block Arnoldi process, or nonsymmetric block Lanczos process to generate low-dimensional Sylvester matrix equations. For larger and sparser problems, iterative methods are more suitable. Effective approaches in such cases include the alternating direction implicit (ADI) method \cite{ref35}, global full orthogonalization method, global generalized minimum residual method, and global Hessenberg and changing minimal residual with Hessenberg process method. In scenarios where system matrices are of low-rank, the ADI method \cite{ref36}, extended block Arnoldi method \cite{ref37}, and global Arnoldi method \cite{ref38,ref39} are found to be effective. The focus of this paper is to utilize several optimization methods to solve the Sylvester matrix equation. To achieve this, we transform the Sylvester matrix equation into an optimization problem and then apply existing optimization methods to solve it.

   In the first part of this paper, we consider transforming the Sylvester matrix equation into the following optimization problems:
   \begin{equation}\label{eq}
   \min_U \left \| U \right \|_{2,1} \; s.t. \; AU = Y,
   \end{equation}
   where $A\in \mathbb{R}^{m \times n}, U \in \mathbb{R}^{n \times s}, Y \in \mathbb{R}^{m \times s}$. The optimization problem (\ref{eq}) is widely utilized in the signal processing community's multiple measurement vector (MMV) model. Existing algorithms often reframe it as a second-order cone programming (SOCP) or semidefinite programming (SDP) problem, which can be solved using interior-point methods or the bundle method. However, SOCP or SDP solutions are computationally expensive, limiting their practical use. Recently, an efficient algorithm was proposed to solve the specific problem (\ref{eq}) through complicated problem reformulation into a min-max problem, followed by application of the proximal method \cite{ref23}. Reported outcomes indicate high efficiency of the algorithm. However, it is a gradient descent type method with slow convergence, and cannot be directly applied to solving other general $\ell_{2,1}$-norm minimization problems. In this paper's first section, we reframe the Sylvester matrix equation as optimization problem (\ref{eq}). Then, we use a straightforward and effective constrained convex optimization method (CCOM) for solving the $\ell_{2,1}$-norm minimization problem. The theoretical analysis guarantees the global convergence of the algorithm, however the matrix equation to be solved by transforming into the vector form by straightening operator, which increases the computational complexity, hence the method is not very efficient actually.

   In the second part of this paper, we consider using classical quasi-Newton methods such as the DFP and BFGS algorithms to solve the Sylvester matrix equation. One remarkable advantage of Newton's method is its local superlinear and even second-order convergence rate, making it a popular choice among algorithms. However, Newton's method may encounter difficulties when the Jacobian matrix is singular, leading to non-existent Newton directions. To address this issue, the quasi-Newton method was proposed by American physicist W. C. Davidon. Different updating formulas for the quasi-Newton matrix give rise to different quasi-Newton methods. Some well-known updating formulas include Broyden's rank-one updating formula (R1) \cite{ref25}, Symmetric rank-one updating formula (SR1) \cite{ref25}, DFP updating formula \cite{ref26}, BFGS updating formula \cite{ref27}, and PSB updating formula \cite{ref24}, which is commonly used in nonlinear least squares problems.

   The BFGS updating formula, a popular quasi-Newtonian formula, possesses various properties of the DFP updating formula. When combined with the Wolfe linear search, the BFGS algorithm exhibits global convergence properties for convex minimization problems. Numerous numerical results have shown that the BFGS method has better numerical performance compared to other quasi-Newton methods. Quasi-Newton methods are highly effective for solving non-linear equations and unconstrained optimization problems. They do not require the direct calculation of the Jacobian matrix while maintaining the fast convergence rate of Newton's method. Since Broyden's pioneering work on the quasi-Newton method in the 1960s for solving nonlinear equations \cite{ref25}, it has quickly gained popularity among optimization researchers and computational mathematicians. In particular, extensive research has been conducted on the local convergence of the quasi-Newton method, which can be found in references \cite{ref24}, \cite{ref27}-\cite{ref29}. Therefore, in the second part of this paper, we transform the Sylvester matrix equation into an unconstrained optimization problem and provide iterative formulas for the DFP and BFGS algorithms to solve the problem. The convergence of the algorithms can be found in \cite{ref22}. For numerical examples, we compare the algorithms with the CG (Conjugate Gradient) method and AR (Anderson-Richardson) method, demonstrating their effectiveness.

   In the next part, we consider the continuous algebraic Riccati equation(CARE) in control theory:
   \begin{equation}\label{eq1.1}
   A^TX + XA - XNX +K = 0,
   \end{equation}
   where $A,N,K \in \mathbb{R}^{n \times n}$ are given matrices, and $X \in \mathbb{R}^{n \times n}$ is an unknown matrix to be solved. Particularly, $N,K$ are symmetric positive semi-definite matrices, and $X$ is symmetric positive semi-definite solution.

   So far, many researchers have been conducted on solving nonlinear matrix equations, resulting in a series of important and valuable findings. In 1985, Yaz, E. \cite{ref8} proposed lower bounds for the eigenvalues of the solution to the unified Riccati equation and highlighted their practical applications. In 1992, Delfour, M. C. and Ouansafi, A. \cite{ref9} introduced a novel non-iterative approximation scheme for solving nonlinear Riccati matrix differential equations. This approach achieved significant computational results without affecting the asymptotic convergence rate of the original model. In 2005, G$\ddot{u}$lsu, M. and Sezer, M. \cite{ref10} presented a matrix-based method for solving Riccati differential equations using Taylor polynomials. The process involves truncating the Taylor series of the function in the equations and substituting their matrix forms into the given equation. Consequently, the resulting matrix equation can be solved, and the unknown Taylor coefficients can be approximated. The solution is obtained in a computationally feasible series form, and numerical results have demonstrated the efficiency of this method for such equation types. In 2018, Ai-Guo Wu et al \cite{ref11} proposed two iterative algorithms for solving stochastic algebraic Riccati matrix equations, which arise in linear quadratic optimal control problems of linear stochastic systems with state-dependent noise. The authors presented certain properties of the sequences generated by these algorithms under appropriate initial conditions, and further analyzed the convergence properties of the proposed methods. Finally, they employed two numerical examples to illustrate the effectiveness of the algorithms they proposed. In 2021, Liu Changli et al \cite{ref12} proved that each smaller area during the doubling iterations is indeed a matrix whose defining matrix is either nonsingular or an irreducible singular M-matrix. Consequently, this matrix possesses a minimal nonnegative solution that can be efficiently computed using the doubling algorithm.

   The purpose of the third and fourth sections of this paper is to formulate the Continuous Algebraic Riccati Equation (CARE) as an optimization problem and utilize the Alternating Direction Multiplier Method (ADMM) to solve it. The classical ADMM is an extended version of the augmented Lagrange multiplier method \cite{ref40}. There is a rich literature on this method in the fields of optimization and numerical analysis. It is commonly used for solving convex programming problems with separable structures, and a comprehensive overview of this method can be found in Chapter 3 of \cite{ref41} and \cite{ref42}. In the third section of this paper, we propose an equivalent constrained optimization formulation for CARE. Subsequently, we investigate the application of the classical ADMM method to solve the constrained optimization problem, providing a unified framework for employing ADMM to solve matrix equations. We also present the convergence analysis and numerical results of the ADMM algorithm. Experimental findings demonstrate that ADMM is an effective optimization algorithm for solving CARE. In the final section of the paper, we transform CARE into a problem of solving a Lyapunov matrix equation using the Newton method. We provide the iterative scheme of the ADMM algorithm specifically designed for solving the Lyapunov equation. Furthermore, we propose the algorithmic framework of Newton-ADMM for solving CARE. By proving the convergence of Newton-ADMM, conducting numerical examples, and comparing the results with the classical Newton method, we verify the efficiency of our proposed algorithm.

   Throughout this paper, we adopt the standard notation in matrix theory. The symbol $I_n$ stand for the identity matrix of order $n$. Let $A^T$ and $\left \| A \right \|_F$ be the transpose and Frobenius norm of the real matrix $A\in \mathbb{R}^{m\times n}$, respectively. The inner product in space $\mathbb{R}^{m\times n}$ is defined as
   \begin{equation*}
   \left \langle A,B \right \rangle :=tr(A^TB)=\sum_{i,j} a_{ij}b_{ij},\;\;\;\forall A,B\in \mathbb{R}^{m\times n}.
   \end{equation*}
   Obviously, $\mathbb{R}^{m\times n}$ is a Hilbert inner product space, and the norm of a matrix generated by this inner product space is the Frobenius norm. 

\section{CCOM Algorithm}\label{sec2}
   In this section, we use the constrained convex optimization method (CCOM) to solve the Sylvester matrix equation with $\ell_{2,1}$-norm. We provide a proof of its convergence and demonstrate its effectiveness through numerical examples. However, it should be noted that transforming the matrix equation into an optimization problem in vector form increases the computational complexity. Consequently, the effectiveness of the algorithm in solving the Sylvester equation is significantly diminished.

   \subsection{Iterative Method}

   First, we transform the Equation (\ref{eq0}) into the following optimization problem:
   \begin{equation}\label{eq27}
   \begin{aligned}
   	\min_{X}  \left \| X \right \|_{2,1} \;
   	s.t. \; AX + XB = C,
   \end{aligned}
   \end{equation}
   where $X\in \mathbb{R}^{m \times n}$.

   Therefore, the Lagrangian function of (\ref{eq27}) is
   \begin{equation*}
   \begin{split}
   	\mathcal{L}(X)= \left \| X \right \|_{2,1} -\left \langle \Lambda,AX+XB-C \right \rangle ,
   \end{split}
   \end{equation*}
   where $\Lambda \in \mathbb{R}^{m \times n} $ is a Lagrangian multiplier.

   The $\ell_{2,1}$-norm of the matrix $A=[a_1^T,a_2^T,...,a_n^T]^T$ is defined as
   \begin{equation*}
   \left \| A \right \|_{2,1}=\sum_{i=1}^n \sqrt{\sum_{j=1}^n \left| A_{i,j} \right|^2}.
   \end{equation*}
   Because of the definition of matrix:
   \begin{equation*}
   \left \| A \right \|_{2,1}=\sum_{i=1}^n \left \| a_i \right \|_2=\sum_{i=1}^n (a_ia_i^T)^{\frac{1}{2}} ,
   \end{equation*}
   then
   \begin{equation*}
   \begin{aligned}
   	\begin{split}
   		\frac{\partial \left \| A \right \|_{2,1}}{\partial A}
   		& =\left( \frac{\partial (\sum_{i=1}^n \left \| a_i \right \|_2)}{\partial a_j } \right)_{n \times 1} = \left( \frac{\partial (\sum_{i=1}^n (a_ia_i^T)^{\frac{1}{2}})}{\partial a_j } \right)_{n \times 1}=\left(  \frac{a_j}{\left \| a_j \right \|_2}  \right)_{n \times 1}\\
   		& =2\begin{pmatrix}
   			\frac{1}{2\left \| a_1 \right \|_2}&  &  &\\
   			& \frac{1}{2\left \| a_2 \right \|_2} &  & \\
   			&  & \ddots &  \\
   			&  &  &  \frac{1}{2\left \| a_n \right \|_2}
   		\end{pmatrix}
   		\begin{pmatrix}
   			a_1\\
   			a_2\\
   			\vdots\\
   			a_n
   		\end{pmatrix}\\
   		& =2\begin{pmatrix}
   			\frac{1}{2\left \| a_1 \right \|_2}&  &  &\\
   			& \frac{1}{2\left \| a_2 \right \|_2} &  & \\
   			&  & \ddots &  \\
   			&  &  &  \frac{1}{2\left \| a_n \right \|_2}
   		\end{pmatrix}A\\
   		& = 2\Sigma A.
   	\end{split}
   \end{aligned}
   \end{equation*}

   Next, by the straightening operatorthe, the Equation (\ref{eq0}) can be rewritten as
   \begin{equation*}
   (I_n \otimes A+B^T \otimes I_m)vec(X)=vec(C).
   \end{equation*}
   Let $M = I_n \otimes A+B^T \otimes I_m$, then we solve the Sylvester matrix equation in the following form:
   \begin{equation}\label{eq28}
   \min_{x}  \left \| x \right \|_{2,1} \;
   s.t. \; Mx = c,
   \end{equation}
   where $x\in \mathbb{R}^{mn \times 1}$.

   The Lagrangian function of the problem in (\ref{eq28}) is
   \begin{equation*}
   \begin{split}
   	\mathcal{L}(X)= \left \| x \right \|_{2,1} -\left \langle \Lambda,Mx-c \right \rangle ,
   \end{split}
   \end{equation*}
   where $\Lambda \in \mathbb{R}^{mn \times 1}$ is a Lagrangian multiplier. Taking the derivative of $\mathcal{L}(X)$ w.r.t $X$, and setting the derivative to zero, we have
   \begin{equation*}
   2Dx-M^T\Lambda=0,
   \end{equation*}
   where $D$ is a diagonal matrix with the $i$-th diagonal element as $d_{ii}=\frac{1}{2\left \| x^i \right \|_2}$.

   Let multiplying the two sides of above equation by $MD^{-1}$, we get
   \begin{equation*}
   2Mx-MD^{-1}M^T\Lambda=0,
   \end{equation*}
   and using the constraint $Mx=c$, we have
   \begin{equation*}
   2c-MD^{-1}M^T\Lambda=0,
   \end{equation*}
   so we can obtain $\Lambda=2(MD^{-1}M^T)^{-1}c$. Finally we arrive at
   \begin{equation}\label{eq29}
   x=D^{-1}M^T(MD^{-1}M^T)^{-1}c.
   \end{equation}
   Therefore, we can get Algorithm \ref{alg2} for solving optimization problem (\ref{eq28}).
   \begin{algorithm}[!t]
   \caption{CCOM}
   \label{alg2}
   \renewcommand{\algorithmicrequire}{\textbf{Input:}}
    \renewcommand{\algorithmicensure}{\textbf{Output:}}
   \begin{algorithmic}
   	\REQUIRE $A, B, C, k=0, kmaxiter$, initialize $D_0=I_n$.
    
   	\STATE 1.When $k \leqslant kmaxiter $, calculate $x_k$ via equation (\ref{eq29}).
   	
   	\STATE 2.Calculate $D_{k+1}$.
   	
   	\STATE 3.if $\left \| AX+XB-C \right \|_F \leqslant \epsilon$, stop iteration.
   	
   	\STATE 4.Let $k \leftarrow k+1$, repeat iteration.
   	
   	\ENSURE $x_k$.
   \end{algorithmic}
   \end{algorithm}

   We analyze the computational complexity of each iteration in Algorithm \ref{alg2}, which is primarily dominated by matrix multiplication and inverse calculation according to Equation (\ref{eq29}). The total computational complexity of Algorithm \ref{alg2} is $11n^6-4n^4$. However, if we utilize LU decomposition to solve the inverse calculation, we can reduce the computational complexity by $\frac{20}{3}n^6-3n^4$.

   \subsection{Convergent Analysis}
   In this section, we focus on the convergence analysis of Algorithm \ref{alg2}. The Algorithm \ref{alg2} monotonically decreases the objective of the problem (\ref{eq28}) in each iteration. To prove it, we need the following lemma.
\begin{lemma}\label{lemma1}
For any nonzero vectors $u,u_t \in \mathbb{R}^c$, the following inequality holds:
\begin{equation*}
\left \| u \right \|_2-\frac{\left \| u \right \|_2^2 }{2\left \| u_t \right \|_2} \leqslant \left \| u_t \right \|_2-\frac{\left \| u_t \right \|_2^2 }{2\left \| u_t \right \|_2}.
\end{equation*}
\end{lemma}
The proof of Lemma \ref{lemma1} can be found in \cite{ref17}.

The convergence of the Algorithm \ref{alg2} is summarized in the following theorem.
\begin{theorem}
The Algorithm \ref{alg2} will monotonically decreases the objective of the problem (\ref{eq28}) in each iteration, and converge to the global optimum of the problem.
\end{theorem}
\begin{proof}
It can easily verified that problem (\ref{eq28}) is equal to the following problem:
\begin{equation*}
\min_{x}  tr(x^TDx) \;
s.t. \; Mx = c,
\end{equation*}
where
\begin{equation*}
x=[x^1,x^2,...,x^n]^T, \;  \; D=\begin{pmatrix}
\frac{1}{2\left \| x^1 \right \|_2}&  &  &\\
 & \frac{1}{2\left \| x^2 \right \|_2} &  & \\
 &  & \ddots &  \\
 &  &  &  \frac{1}{2\left \| x^n \right \|_2}
\end{pmatrix}.
\end{equation*}
Thus in the $t$-th iteration,
\begin{equation*}
x_{t+1}=\mathop{\arg\min}\limits_{x;Mx=c}tr(x^TD_tx),
\end{equation*}
hence $tr(x_{t+1}^TD_tx_{t+1}) \leqslant tr(x_t^TD_tx_t)$, which indicates that
\begin{equation}\label{eq30}
\sum_{i=1}^n \frac{\left \| x_{t+1}^i \right \|_2^2}{2\left \| x_{t}^i \right \|_2} \leqslant \sum_{i=1}^n \frac{\left \| x_{t}^i \right \|_2^2}{2\left \| x_{t}^i \right \|_2}.
\end{equation}
On the other hand, according to Lemma \ref{lemma1}, for each $i$ we have
\begin{equation*}
\left \| x_{t+1}^i \right \|_2 - \frac{\left \| x_{t+1}^i \right \|_2^2}{2\left \| x_{t}^i \right \|_2} \leqslant \left \| x_{t}^i \right \|_2 - \frac{\left \| x_{t}^i \right \|_2^2}{2\left \| x_{t}^i \right \|_2},
\end{equation*}
thus the following inequality holds
\begin{equation}\label{eq31}
\sum_{i=1}^n(\left \| x_{t+1}^i \right \|_2 - \frac{\left \| x_{t+1}^i \right \|_2^2}{2\left \| x_{t}^i \right \|_2}) \leqslant \sum_{i=1}^n(\left \| x_{t}^i \right \|_2 - \frac{\left \| x_{t}^i \right \|_2^2}{2\left \| x_{t}^i \right \|_2}).
\end{equation}
Combining Equation (\ref{eq30}) and Equation (\ref{eq31}), we arrive at
\begin{equation*}
\sum_{i=1}^n \left \| x_{t+1}^i \right \|_2  \leqslant \sum_{i=1}^n \left \| x_{t}^i \right \|_2 .
\end{equation*}
\end{proof}
That is to say, $\left \| x_{t+1} \right \|_{2,1} \leqslant \left \| x_{t}\right \|_{2,1}$. Thus, Algorithm \ref{alg2} will monotonically decrease the objective of problem (\ref{eq28}) in each iteration tt. As problem (\ref{eq28}) is a convex problem, satisfying Equation (\ref{eq29}) indicates that $X$ is a global optimum solution to problem (\ref{eq28}). Therefore, Algorithm \ref{alg2} will converge to the global optimum of problem (\ref{eq28}).

\subsection{Experimental Results}
In this section, we present two corresponding numerical examples to demonstrate the efficiency of Algorithm \ref{alg2}. All code is implemented in Matlab. The iteration step and error of the objective function are denoted as the iteration and error, respectively. We set the matrix order "$n$" to be 10, 100, and 200. Initially, we initialize $X_0$ as an null matrix, and error precision is $\epsilon=10^{-8}$. We use the error of matrix equation as
\begin{equation*}
\left \| R_k \right \|_F=\left \| AX_k + X_kB - C \right \|_F.
\end{equation*}
\subsubsection{Example 1}
We will compare the efficiency of Algorithm \ref{alg2} by increasing the order of the following given coefficient matrices,
\begin{equation*}
A =
\begin{pmatrix}
2 & -4 & & & & \\
-4 & 2 & -4 & &  &\\
 & \ddots & \ddots & \ddots &\\
 & & \ddots & \ddots & -4 &\\
 & & & -4& 2& \\
\end{pmatrix}_{n \times n},
B =
\begin{pmatrix}
1 & 3 & & & & \\
3 & 1 & 3 & &  &\\
 & \ddots & \ddots & \ddots &\\
 & & \ddots & \ddots & 3 &\\
 & & & 3 & 1& \\
\end{pmatrix}_{n \times n},
\end{equation*}
\begin{equation*}
C=I_n.
\end{equation*}
When considering different matrix orders, the error data is displayed in Table 1. After reducing the computational complexity, the error data is shown in Table 2. We can not increase the coefficient matrix order to 300 due to limitations in computer memory. Additionally, even with the reduction in computational complexity, the calculation rate remains quite slow.
\begin{table}[!htbp]\label{tab1}
	\centering
 \caption{Numerical results for Example 1}
	\begin{tabular}{cccc}
	\hline
	n & iteration & error & time(s) \\ \hline
	10 & 1 & 6.0905e-13 & 0.05 \\
	100 & 1 & 1.1574e-09 & 22.8 \\
	200 & 2 & 1.9798e-09 & 2288 \\
\hline
	\end{tabular}
	\end{table}

\begin{table}[!htbp]\label{tab2}
	\centering
 \caption{Numerical results for Example 1}
	\begin{tabular}{cccc}
	\hline
	n & iteration & error & time(s) \\ \hline
	10 & 1 & 6.0905e-13 & 0.01 \\
	100 & 1 & 1.1574e-09 & 19 \\
	200 & 2 & 1.9798e-09 & 1860 \\
\hline
	\end{tabular}
	\end{table}
	
\subsubsection{Example 2}
In this example, we consider coefficient matrices that are diagonal. We compare the computational complexity of Algorithm \ref{alg2} for different orders of the coefficient matrices. The coefficient matrices are as follows:
\begin{equation*}
A =
\begin{pmatrix}
2 & 0 & & & & \\
0 & 2 & 0 & &  &\\
 & \ddots & \ddots & \ddots &\\
 & & \ddots & \ddots & 0 &\\
 & & & 0& 2& \\
\end{pmatrix}_{n \times n},
B =
\begin{pmatrix}
1 & 0 & & & & \\
0 & 1 & 0 & &  &\\
 & \ddots & \ddots & \ddots &\\
 & & \ddots & \ddots & 0 &\\
 & & & 0 & 1& \\
\end{pmatrix}_{n \times n},
\end{equation*}
\begin{equation*}
C=I_n.
\end{equation*}
We have the error data in Table 3.
\begin{table}[!htbp]\label{tab3}
	\centering
 \caption{Numerical results for Example 2}
	\begin{tabular}{cccc}
	\hline
	n & iteration & error & time(s) \\ \hline
	10 & 1 & 1.2560e-13 & 0.002 \\
	100 & 1 & 2.3124e-09 & 14 \\
	200 & 1 & 4.9651e-09 & 640 \\
\hline
	\end{tabular}
	\end{table}
\subsubsection{Analysis of numerical results}
The numerical results indicate that the CCOM algorithm is effective in solving Sylvester matrix equations. However, its use of the straightening operator increases the computational complexity, making it difficult to handle large order coefficient matrices. While the number of iterative steps is small, the execution time is relatively long. Furthermore, we observe that different sparsity patterns in the coefficient matrices lead to varying outcomes when evaluating different numerical examples. Evidently, the second numerical example yields better results. Therefore, there is a need to enhance the CCOM algorithm for more efficient resolution of Sylvester matrix equations.
\section{Quasi-Newton Method}
In this section, we employ two Quasi-Newton methods about DFP algorithm and BFGS algorithm to solve the Sylvester matrix equation. First, we introduce the basic theory of Quasi-Newton method. Then we transform the solution of the Sylvester matrix equation into an optimization problem, and give the iterative scheme of the DFP algorithm and BFGS algorithm for solving the Sylvester matrix equation, also we provide the convergence and complexity analysis of the algorithm. Finally, we present some numerical experiments and compare with other algorithms to illustrate the efficiency of DFP algorithm and BFGS algorithm for solving Sylvester matrix equation.

\subsection{Preliminaries}
We introduce the basic theory of Quasi-Newton method primarily. In general, we use Newton's method to solve unconstrained optimization $\min_{x} f(x)$ due to the condition $\triangledown f(x)=0$ as
\begin{equation*}
x_{k+1}=x_k-H_k^{-1}g_k,
\end{equation*}
where $g_k=g(x_k)=\triangledown f(x_k)$ is the derivative value of $f(x)$ at $x_k$, and $H_k=H(x_k)$, $H(x)$ is the  Hessian matrix of $f(x)$. One disadvantage of Newton's method is that has to calculate $H_k^{-1}$ in each iteration, it's very complex, and sometimes the Hessian matrix of the objective function cannot be positive definite, which makes the Newton method invalid. In order to overcome these problems, researchers present the Quasi-Newton method. The basic concept of this method is to construct a positive definite symmetric matrix that can approximate the Hessian matrix (or its inverse matrix) without calculating two-order partial derivative.

By using the Taylor expansion of $\triangledown f(x)$, we obtain that
\begin{equation*}
\triangledown f(x) = g_k+H_k(x-x_k).
\end{equation*}
Let $x=x_{k-1}$, then
\begin{equation*}
g_{k-1}- g_k=H_k(x_{k-1}-x_k),
\end{equation*}
we note $y_k=g_k-g_{k-1}, \delta_k=x_k-x_{k-1}$, and $y_{k-1}=H_k\delta_{k-1}$, or $H_k^{-1}y_{k-1}=\delta_{k-1}$.

If $H_k$ is a positive definite matrix, then the search direction $p_k=-H_k^{-1}g_k$ is descending, so the Newton's method requires that the Hessian matrix must be positive definite. Quasi-Newton methods are taking the $G_k$ as an approximation of $H_k^{-1}$, so $G_k$ needs to satisfy the same condition, i.e. $G_k$ must be positive definite and $G_{k+1}y_k=\delta_k$, then the iterative scheme of $G_k$ is
\begin{equation*}
G_{k+1}=G_k+\vartriangle G_k.
\end{equation*}
In addition, there are two types of inexact linear search in the Quasi-Newton algorithm generally as Armijo linear search and Wolfe-Prowell linear search. The Armijo linear search must satisfy the following conditions, for $\sigma_1 \in (0,1/2)$, $\alpha_k>0$ then
\begin{equation*}
f(x_k+\alpha_k d_k) \leqslant f(x_k)+\sigma_1 \alpha_k  \triangledown f(x)^Td_k,
\end{equation*}
where $d_k$ is descent direction and $\alpha_k$ is iteration step.

The Wolfe-Prowell linear search satisfy that given the $\sigma_1, \in (0,1/2),\sigma_1<\sigma_2<1$, and $\alpha_k>0$, we have
\begin{equation}\label{eq35}
\begin{cases}
f(x_k+\alpha_k d_k) \leqslant f(x_k)+\sigma_1 \alpha_k  \triangledown f(x)^Td_k,\\
\triangledown f(x+\alpha_k d_k)^Td_k \geqslant \sigma_2 \triangledown f(x)^Td_k,
\end{cases}
\end{equation}
where $d_k$ is descent direction and $\alpha_k$ is iteration step, the second condition is limiting too small step size.  We prove the feasibility of Wolfe-Prowell linear search through the convergence of the algorithm and numerical results in the next section, but use Armijo linear search cannot guarantee the efficiency of our algorithm.

Next, we will transform the solution of the Sylvester matrix equation into two optimization problems, and we will prove the equivalence of the two optimization problems for providing the corresponding iterative scheme and convergence of the algorithm.

\begin{itemize}
\item  We can rewrite equation (\ref{eq0}) as
\begin{equation}\label{eq33}
\min{f_1(X)}  =\min \frac{1}{2} \left \| AX+XB-C \right \|_F^2,
\end{equation}
where $X$ is the independent variable, due to the $f_1(X)$ is a convex function, thus the optimization problem has a unique solution, hence optimization problem (\ref{eq33}) has the same solution with equation (\ref{eq0}).

\item  On the other hand, we arrive the following form by using straightening operator,
\begin{equation*}
(I_n \otimes A+B^T \otimes I_m)vec(X)=vec(C).
\end{equation*}
Let $I_n \otimes A+B^T \otimes I_m=M$, and $vec(X)=x$, $vec(C)=c$, we have
\begin{equation}\label{eq34}
\min{f_2(x)}  =\min \frac{1}{2} \left \| Mx-c \right \|_2^2,
\end{equation}
where the independent variable is $x\in \mathbb{R}^{mn \times 1}$, and the optimization problem (\ref{eq34}) also has the same solution with equation (\ref{eq0}).
\end{itemize}

Let $D=[d_{ij}]\in \mathbb{R}^{m \times n},i=1,2,...,m; j=1,2,...,n$, then
\begin{equation*}
\left \| D \right \|_F^2 = \sum_{i=1}^m \sum_{j=1}^n \left| d_{ij} \right|^2 = \left \| vec(D) \right \|_2^2.
\end{equation*}
Therefore, we can obtain that optimization problem (\ref{eq33}) and (\ref{eq34}) are equivalent due to the definition of Frobenius norm and vector 2-norm. Next, we will give the iterative scheme of the algorithm from optimization problem (\ref{eq33}), and prove the convergence of the algorithm from optimization problem (\ref{eq34}).

\subsection{Iterative Methods}
In this section, we provide the iterative scheme of the Hessian matrix (or its inverse) and the computational complexity about the DFP algorithm and BFGS algorithm for solving the optimization problem (\ref{eq33}).
\subsubsection{DFP Algorithm}
We assume that the $G_k$ is the approximate matrix of the inverse of the Hessian matrix in each iteration, and set $G_{k+1}$ as
\begin{equation*}
G_{k+1}=G_k+P_k+Q_k,
\end{equation*}
where $P_k$ and $Q_k$ is undetermined, thus
\begin{equation*}
G_{k+1}y_k=G_ky_k+P_ky_k+Q_ky_k.
\end{equation*}
Let $P_k$ and $Q_k$ satisfy that
\begin{equation*}
P_ky_k=\delta_k, \; Q_ky_k=-G_ky_k,
\end{equation*}
where $y_k=g(X_k)-g(X_{k-1})$, $\delta_k=X_k-X_{k-1}$.

Therefore,
\begin{equation*}
P_k=\frac{\delta_k\delta_k^T}{\delta_k^Ty_k}, \; Q_k=-\frac{G_ky_ky_k^TG_k}{y_k^TG_ky_k}.
\end{equation*}

However, we find that the denominator of $P_k$ and $Q_k$ is a matrix. In order to ensure the effectiveness of the iterative scheme, it is crucial to guarantee the invertibility of the denominator matrix. In cases where the matrix is nonsingular, we can consider utilizing the pseudo inverse. Similarly, the BFGS algorithm follows the same principle, and the iterative scheme can be enhanced accordingly. The DFP algorithm for solving the optimization problem (\ref{eq33}) is outlined in Algorithm \ref{alg3}.
\begin{algorithm}
\caption{DFP}
\label{alg3}
\renewcommand{\algorithmicrequire}{\textbf{Input:}}
\renewcommand{\algorithmicensure}{\textbf{Output:}}
\begin{algorithmic}
\REQUIRE $f_1(X)$, $g(X)= \triangledown f(X)$, $\epsilon$.

\STATE 1.Initialize $X_0$, set $G_0=I_n$ and $k=0$.

\STATE 2.Calculate $g_k=g(X_k)$, if $\left \| g_k \right \|<\epsilon$, stop iteration and get $X^*=X_k$, or go to step 3.

\STATE 3.Set $p_k=-G_kg_k$.

\STATE 4.Solve $\lambda_k$ as
\begin{equation*}
f(X_k+\lambda_kp_k)=\min_{\lambda \geqslant 0}f(X_k+\lambda p_k).
\end{equation*}
\STATE 5.Set $X_{k+1}=X_k+\lambda_kp_k$.

\STATE 6.Calculate $g_{k+1}=g(X_{k+1})$, if $\left \| g_{k+1} \right \|<\epsilon$, stop iteration and get $X^*=X_{k+1}$, or
\begin{equation*}
G_{k+1}=G_k+\frac{\delta_k\delta_k^T}{\delta_k^Ty_k}-\frac{G_ky_ky_k^TG_k}{y_k^TG_ky_k}.
\end{equation*}

\STATE 7.Set $k=k+1$, go to step 3.

\ENSURE Minimum point $X^*$ of $f_1(X)$.

\end{algorithmic}
\end{algorithm}
\subsubsection{BFGS Algorithm}
The classical BFGS algorithm set $B_k$ as the approximation of Hessian matrix $H_k$, so we directly give the iterative formula as
\begin{equation*}
B_{k+1}=B_k+\frac{y_ky_k^T}{y_k^T\delta_k}-\frac{B_k\delta_k\delta_k^TB_k}{\delta_k^TB_k\delta_k}.
\end{equation*}
If let $G_k=B_k^{-1}$, we have
\begin{equation*}
G_{k+1}=(I-\frac{\delta_ky_k^T}{\delta_k^Ty_k})G_k(I-\frac{\delta_ky_k^T}{\delta_k^Ty_k})^T+\frac{\delta_k\delta_k^T}{\delta_k^Ty_k}.
\end{equation*}
Similarly, The BFGS algorithm for solving optimization problem (\ref{eq33}) is summarized in Algorithm \ref{alg4}.
\begin{algorithm}[H]
\caption{BFGS}
\label{alg4}
\renewcommand{\algorithmicrequire}{\textbf{Input:}}
\renewcommand{\algorithmicensure}{\textbf{Output:}}
\begin{algorithmic}
\REQUIRE $f_1(X)$, $g(X)= \triangledown f(X)$, $\epsilon$.

\STATE 1.Initialize $X_0$, set $G_0=I_n$ and $k=0$.

\STATE 2.Calculate $g_k=g(X_k)$, if $\left \| g_k \right \|<\epsilon$, stop iteration and get $X^*=X_k$, or go to step 3.

\STATE 3.Set $p_k=-G_kg_k$.

\STATE 4.Solve $\lambda_k$ as
\begin{equation*}
f(X_k+\lambda_kp_k)=\min_{\lambda \geqslant 0}f(X_k+\lambda p_k).
\end{equation*}
\STATE 5.Set $X_{k+1}=X_k+\lambda_kp_k$.

\STATE 6.Calculate $g_{k+1}=g(X_{k+1})$, if $\left \| g_{k+1} \right \|<\epsilon$, stop iteration and get $X^*=X_{k+1}$, or
\begin{equation*}
G_{k+1}=(I-\frac{\delta_ky_k^T}{\delta_k^Ty_k})G_k(I-\frac{\delta_ky_k^T}{\delta_k^Ty_k})^T+\frac{\delta_k\delta_k^T}{\delta_k^Ty_k}.
\end{equation*}

\STATE 7.Set $k=k+1$, go to step 3.

\ENSURE Minimum point $X^*$ of $f_1(X)$.

\end{algorithmic}
\end{algorithm}

We analyze the computational complexity in each iteration of Algorithm \ref{alg3} and Algorithm \ref{alg4}. It's mainly controlled by matrix multiplication with the first-order derivatives as $g(X_k)=A^TAX+XBB^T+\frac{1}{2}(A^TX+AX)B+\frac{1}{2}B^T(A^TX+AX)-(A^TC+CB^T)$ and the iterative formula of $G_k$ . Due to independent variable $X\in \mathbb{R}^{m\times n}$, the total computational complexity is $4m^3+2n^3+10n^2m+16m^2n-4m^2-2n^2-9mn$ in Algorithm \ref{alg3}, and the total computational complexity is $8m^3+n^3+6n^2m+14m^2n-6m^2-8mn $ in Algorithm \ref{alg4}.

\subsection{Convergent Analysis}
In this section, we focus on the convergence analysis of Algorithm \ref{alg3} and Algorithm \ref{alg4}. Since the classical DFP algorithm and BFGS algorithm are aimed at solving optimization problems with vector independent variables, we prove the convergence of the algorithm based on the optimization problem (\ref{eq34}), which is equivalent to optimization problem (\ref{eq33}). The convergence of DFP algorithm requires stronger conditions than BFGS algorithm, hence we focus on the convergence of algorithm \ref{alg4}. According to the paper \cite{ref18}, the inexact Armijo linear search cannot guarantee the global convergence of the algorithm, so we only prove the convergence with inexact Wolfe-Prowell linear search.

\subsubsection{Convergence of Algorithm \ref{alg4}}
In terms of the equivalence of \ref{eq33} and \ref{eq34}, we focus on convergent analysis with $f_2(x)$. The classic BFGS algorithm is an effective method for solving the following optimization problems
\begin{equation*}
\min_{x\in R^n} f(x),
\end{equation*}
where $B_k$ is the approximation of Hesse matrix $H_k$. We arrive $B_{k+1}$ by low rank update of $B_k$, then
\begin{equation}\label{eq36}
B_{k+1}=B_k+\vartriangle_k,
\end{equation}
where $\vartriangle_k$ is rank 2. Let $\vartriangle_k=a_ku_ku_k^T+b_kv_kv_k^T$ in (\ref{eq36}), $a_k,b_k$ are real numbers to be solved, and $u_k,v_k \in \mathbb{R}^n$ are vectors to be solved. According to the condition $y_{k-1}=H_k\delta_{k-1}$, we have
\begin{equation*}
B_k \delta_k+a_k(u_k^T\delta_k)u_k+b_k(v_k^T\delta_k)v_k=y_k,
\end{equation*}
or
\begin{equation}\label{eq37}
a_k(u_k^T\delta_k)u_k+b_k(v_k^T\delta_k)v_k=y_k-B_k \delta_k.
\end{equation}
Set $u_k$ and $v_k$ parallel to $B_k\delta_k$ and $y_k$ respectively, thus $u_k=\beta_kB_k\delta_k, v_k=\gamma_ky_k$, where $\beta_k$ and $\gamma_k$ are undetermined parameters, we arrive
\begin{equation*}
\vartriangle_k=a_k\beta_k^2B_k\delta_k\delta_k^TB_k+b_k\gamma_k^2y_ky_k^T.
\end{equation*}
According to (\ref{eq37}), then
\begin{equation*}
[a_k\beta_k^2(\delta_k^T B_k\delta_k)+1]B_k\delta_k+[b_k\gamma_k^2(y_k^T\delta_k)-1]y_k=0.
\end{equation*}
If $y_k$ is not parallel to $B_k\delta_k$ ,then
\begin{equation*}
a_k\beta_k^2=-\frac{1}{\delta_k^T B_k\delta_k}, \; b_k\gamma_k^2=\frac{1}{y_k^T\delta_k},
\end{equation*}
thus
\begin{equation*}
\vartriangle_k=\frac{y_ky_k^T}{y_k^T\delta_k}-\frac{B_k\delta_k\delta_k^TB_k}{\delta_k^TB_k\delta_k},
\end{equation*}
we can obtain that
\begin{equation}\label{eq38}
B_{k+1}=B_k+\frac{y_ky_k^T}{y_k^T\delta_k}-\frac{B_k\delta_k\delta_k^TB_k}{\delta_k^TB_k\delta_k}.
\end{equation}
Obviously, if $B_k$ is symmetric, then $B_{k+1}$ also is symmetric. It can be verified that if $B_0$ is a positive definite matrix, then each matrix $B_k$ of the iteration is a positive definite matrix.

To prove the convergence of Algorithm \ref{alg4}, we need the following lemma \cite{ref18}.
\begin{lemma}
Set $B_k$ is a symmetric positive definite matrix, $B_{k+1}$ is given matrix by (\ref{eq38}). If and only if $y_k^T\delta_k>0$, $B_{k+1}$ is a symmetric positive definite matrix.
\end{lemma}
The condition of $y_k^T\delta_k>0 \; (\forall k \geqslant 0)$ can be attained by the following lemma.
\begin{lemma}
Let $d_k$ satisfy $\triangledown f(x_k)^Td_k<0$. If one of the following conditions is true, then $y_k^T\delta_k>0 \; (\forall k \geqslant 0)$.

(1)Wolfe-Powell linear search is adopted in the algorithm;

(2)The function $f$ is twice continuously differentiable and for all $x\in \mathbb{R}^n$, we have $\triangledown^2 f(x)$ is positive definite.
\end{lemma}
This lemma shows that if $B_0$ is a symmetric positive definite, then the sequence of matrices ${B_k}$ generated by the BFGS algorithm with Wolfe-Powell linear search will also be symmetric positive definite. This property holds true regardless of the type of linear search used, and it remains valid when solving the minimal problem of a uniformly convex function.

We set $G_k=B_k^{-1},G_{k+1}=B_{k+1}^{-1}$, then
\begin{equation*}
G_{k+1}=(I-\frac{\delta_ky_k^T}{\delta_k^Ty_k})G_k(I-\frac{\delta_ky_k^T}{\delta_k^Ty_k})^T+\frac{\delta_k\delta_k^T}{\delta_k^Ty_k}.
\end{equation*}
Therefore, we prove the convergence of Algorithm \ref{alg4} for solving Sylvester equation by the above iterative formula.

Further, we give the global convergence of Algorithm \ref{alg4} \cite{ref18}, and need the following assumption.
\begin{assume}\label{ass1}
(1) The function $f:\mathbb{R}^n \to \mathbb{R}$ is twice continuous differentiable;

$\; \; \; \; \; \; \; \; \; \; \; \;$(2) Level set
\begin{equation*}
\Omega(x_0)=\{x\in \mathbb{R}^n |f(x) \leqslant f(x_0)\},
\end{equation*}
is a bounded convex set, and $f$ is a uniform convex function on $\Omega(x_0)$. So there exist a constant $m \leqslant M$, then
\begin{equation*}
m \left \| d \right \|^2 \leqslant d^T \triangledown^2 f(x) d \leqslant M \left \| d \right \|^2, \; \; ( \; \forall x \in \Omega(x_0), \; d\in R^n).
\end{equation*}
\end{assume}
Under the condition of Assume \ref{ass1}, we have the following lemma
\begin{lemma}
If Assume \ref{ass1} is met, the following
\begin{equation*}
\left\{ \frac{\left \| y_k \right \|}{\left \| \delta_k \right \|} \right\}, \; \left\{ \frac{\left \| \delta_k \right \|}{\left \| y_k \right \|} \right\}, \; \left\{ \frac{ y^T_k \delta_k }{\left \| \delta_k \right \|^2} \right\}, \; \left\{ \frac{ y^T_k \delta_k }{\left \| y_k \right \|^2} \right\}, \; \left\{ \frac{ \left \| y_k  \right \|^2}{y^T_k \delta_k} \right\}, \;
\end{equation*}
are bounded sequences.
\end{lemma}

\begin{lemma}
If Assume \ref{ass1} is met, there exist a constant $C>0$, so that
\begin{equation}\label{eq41}
tr(B_k) \leqslant Ck \; \;( \forall k \geqslant 0 ),
\end{equation}
\begin{equation}\label{eq42}
\frac{1}{k+1} \sum_{i=0}^k \frac{\left \| B_i\delta_i \right \|^2}{\delta_i^TB_i\delta_i} \leqslant C \; (\forall k \geqslant 0),
\end{equation}
where $tr(A)$ denote the trace of matrix $A$.
\end{lemma}

\begin{lemma}
If Assume \ref{ass1} is met, and ${x_k}$ is the sequence generated by Algorithm \ref{alg4} with Wolfe-Powell linear search, there exist a constant $C_2>0$, then
\begin{equation}\label{eq43}
\left \| d_k \right \| \leqslant C_2\alpha_k^{-1}\left \| \triangledown f(x_k) \right \| \cos(\theta_k),
\end{equation}
where $\theta_k$ denote the angle between $d_k$ and $-\triangledown f(x_k)$, and there exist a constant $\alpha>0$, so that $\alpha_k$ satisfy
\begin{equation}\label{eq44}
\prod_{i=0}^k \alpha_i \geqslant \alpha_{k+1}, \; (\forall k \geqslant 0).
\end{equation}
\end{lemma}
Thus we have the global convergence theorem of Algorithm \ref{alg4} \cite{ref18}.
\begin{theorem}
If Assume \ref{ass1} is met, and ${x_k}$ is the sequence generated by Algorithm \ref{alg4} with Wolfe-Powell linear search, then converge to the unique minimum point $x^*$ of the optimization problem (\ref{eq34}).
\end{theorem}

\subsubsection{Convergence of Algorithm \ref{alg3}}
In 1971, Powell \cite{ref19} proved that if the objective function is convex, then the DFP algorithm will converge with exact linear search. However, for inexact linear search, we can only obtain the convergence of the BFGS algorithm. Convergence of the DFP algorithm under inexact line search has always been a challenging problem in the field of nonlinear optimization. For a long time, there was no result even when the objective function was uniformly convex. In 1997, Dachuan Xu \cite{ref20} proved the global convergence of the DFP algorithm under Wolfe-Powell linear search with a uniformly convex objective function under certain assumptions. The following is a reiteration of the assumptions, lemmas, and theorems presented in \cite{ref20}. In terms of the equivalence of (\ref{eq33}) and (\ref{eq34}), our focus is on analyzing the convergence with $f_2(x)$ as follows:
\begin{equation}\label{eq47}
G_{k+1}=G_k+\frac{\delta_k\delta_k^T}{\delta_k^Ty_k}-\frac{G_ky_ky_k^TG_k}{y_k^TG_ky_k},
\end{equation}
where the independent variable is $x\in \mathbb{R}^{mn}$.
\begin{assume}\label{ass2}
Let $f(x)$ be a uniform convex function with second order continuous differentiable, and exist constants $M$ and $m$, $M>m>0$ so that
\begin{equation*}
m \left \| x \right \|^2 \leqslant x^T \triangledown^2 f(y) x \leqslant M \left \| x \right \|^2, \; \; (\forall x,y\in R^n),
\end{equation*}
where $\triangledown^2 f(y)$ is the Hessian matrix of $f(x)$ at $y$.
\end{assume}

\begin{assume}\label{ass3}
When $k$ is large enough, $\frac{g_k^TG_kg_k}{u_k}$ monotonously decreased, where $u_k=\frac{-\delta_k^Tg_k}{\delta_ky_k}$.
\end{assume}

The following Lemma \ref{lem7} - Lemma \ref{lem9} are repeated in the paper \cite{ref21}:
\begin{lemma}\label{lem7}
$\sum_{k=1}^\infty -\delta_k^Tg_k=\sum_{k=1}^\infty \alpha_k g_k^TG_kg_k<\infty$.
\end{lemma}
\begin{lemma}\label{lem8}
There is an inequality holds: $-(1-\sigma)\delta_k^Tg_k \leqslant \delta_k^Ty_k \leqslant -\frac{2(1-\delta)-M_1}{m}\delta_k^Tg_k $.
\end{lemma}
\begin{lemma}\label{lem9}
There exist a constant $C_1>0$, so that
\begin{equation}
\sum_{i=1}^k \frac{1}{\alpha_i} \leqslant C_1k, \;\;\; \sum_{i=1}^k \alpha_i \geqslant k/C_1.
\end{equation}
\end{lemma}

The following lemma is from paper \cite{ref20}.
\begin{lemma}\label{lem10}
$\lim_{k \to \infty}g_k^TG_kg_k=0$.
\end{lemma}

Finally, based on the aforementioned assumptions and lemmas \cite{ref20}, the global convergence theorem for the DFP algorithm is provided.
\begin{theorem}
If Assume \ref{ass2} and Assume \ref{ass3} are met, and $x_k$ is the sequence generated by Algorithm \ref{alg3}. Set $\sigma<(m/M)^3$, then converge to the unique minimum point $x^*$ of the optimization problem (\ref{eq34}).
\end{theorem}

\subsection{Experimental Results}
In this section, we present three corresponding numerical examples to illustrate the efficiency of Algorithm \ref{alg3} and Algorithm \ref{alg4}. Additionally, we compare some effective algorithms for solving the Sylvester matrix equation and conclude that Algorithm \ref{alg3} and \ref{alg4} are more efficient for solving the Sylvester matrix equation when using the Wolfe-Powell linear search. All code is written in Matlab language. We denote iteration and error by the iteration step and error of the objective function, respectively. We take the matrix order "$n$" as 128, 256, 512, 1024, 2048 and 4096, initialize $X_0$ as an null matrix, and set the error precision is $\epsilon=10^{-8}$. We use the error of matrix equation as
\begin{equation*}
\left \| R_k \right \|_F=\left \| AX+XB-C \right \|_F.
\end{equation*}

The following are the numerical results of Algorithm \ref{alg3} and Algorithm \ref{alg4} to solve the optimization problem (\ref{eq33}). Section \ref{3.4.1} is the numerical results of using Armijo linear search; Section \ref{3.4.2} is the numerical results of Wolfe-Prowell linear search.  Based on the convergence analysis and numerical results, it can be observed that the second inexact linear search is more effective. In Section \ref{3.4.3}, a comparison of the numerical results for the two algorithms proposed in this paper, namely CG algorithm and AR algorithm, is provided for solving the same Sylvester matrix equation. Through the analysis of the numerical results, we conclude that the two algorithms proposed in this paper are more efficient.

\subsubsection{Example 1}\label{3.4.1}
We compare Algorithm \ref{alg3} and Algorithm \ref{alg4} with Armijo linear search, then the given coefficient matrices are
\begin{equation*}
A =
\begin{pmatrix}
2 & -1 & & & & \\
-1 & 2 & -1 & &  &\\
 & \ddots & \ddots & \ddots &\\
 & & \ddots & \ddots & -1 &\\
 & & & -1& 2& \\
\end{pmatrix}_{n \times n}, \; \; \;
B =
\begin{pmatrix}
4 & 1 & & & & \\
1 & 4 & 1 & &  &\\
 & \ddots & \ddots & \ddots &\\
 & & \ddots & \ddots & 1 &\\
 & & & 1 & 4& \\
\end{pmatrix}_{n \times n},
\end{equation*}

\begin{equation*}
C=I_n.
\end{equation*}
Where $G_0=I_n$, when the matrix order is different, the error data is shown in Table 4. In this numerical example, Algorithm \ref{alg4} can only reduce the error to $10^{-2}$. Therefore, it is confirmed that using Armijo-type linear search cannot guarantee the convergence of our algorithms. In terms of iteration steps and CPU time, Algorithm \ref{alg3} exhibits poor computational performance in this example. As the matrix order increases, Algorithm \ref{alg3} incurs a high computational cost.

\begin{table}[!htpb]
	\label{tab4}
	\centering
 \caption{Numerical results for Example 1}
	\begin{tabular}{ccccc}
	\hline
	algorithm & n & iteration & error & time(s) \\ \hline
	DFP & 128 & 316 & 9.4577e-08 & 5.6 \\
	DFP & 256 & 287 & 4.8782e-07 & 22 \\
	DFP & 512 & 275 & 9.6180e-07 & 170 \\
	DFP & 1024 & 246 & 4.9609e-07 & 1815 \\
\hline
	\end{tabular}
	\end{table}

\subsubsection{Example 2}\label{3.4.2}
We compare Algorithm \ref{alg3} and Algorithm \ref{alg4} with Wolfe-Prowell linear search, then the given coefficient matrices are
\begin{equation*}
A =
\begin{pmatrix}
3 & -2 & & & & \\
-2 & 3 & -2 & &  &\\
 & \ddots & \ddots & \ddots &\\
 & & \ddots & \ddots & -2 &\\
 & & & -2& 3& \\
\end{pmatrix}_{n \times n},\; \;
B =
\begin{pmatrix}
6 & 2 & & & & \\
2 & 6 & 2 & &  &\\
 & \ddots & \ddots & \ddots &\\
 & & \ddots & \ddots & 2 &\\
 & & & 2 & 6& \\
\end{pmatrix}_{n \times n},
\end{equation*}

\begin{equation*}
C=I_n,
\end{equation*}
where $G_0=I_n$. We also compared our algorithms with the AR algorithm. As the coefficient matrix is not symmetric positive definite, it cannot be directly compared with the CG algorithm. Hence, we present the error data for different matrix orders in Table 5.
\begin{table}[!htbp]
	\label{tab5}
	\centering
 \caption{Numerical results for Example 2}
	\begin{tabular}{ccccc}
	\hline
	algorithm & n & iteration & error & time(s) \\ \hline
	DFP & 128 & 3 & 9.3259e-15 & 0.04 \\
	BFGS & 128 & 3 & 1.5774e-16 & 0.04 \\
	AR & 128 & 3 & 1.1102e-16 & 0.09 \\ \hline
	DFP & 256 & 3 & 9.3259e-15 & 0.21 \\
	BFGS & 256 & 3 & 1.5774e-16 & 0.22 \\
	AR & 256 & 3 & 1.1102e-16 & 0.18 \\ \hline
	DFP & 512 & 3 & 9.3259e-15 & 0.98 \\
	BFGS & 512 & 3 & 1.5774e-16 & 1.02 \\
	AR & 512 & 3 & 1.1102e-16 & 1.03 \\	\hline
	DFP & 1024 & 3 & 9.3259e-15 & 4.87 \\
	BFGS & 1024 & 3 & 1.5774e-16 & 4.87 \\
	AR & 1024 & 3 & 1.1102e-16 & 8.59 \\ \hline
	DFP & 2048 & 3 & 9.3259e-15 & 31 \\
	BFGS & 2048 & 3 & 1.5774e-16 & 33 \\
	AR & 2048 & 3 & 1.1102e-16 & 73 \\ \hline
	DFP & 4096 & 3 & 9.3259e-15 & 191 \\
	BFGS & 4096 & 3 & 1.5774e-16 & 196 \\
	AR & 4096 & 3 & 1.1102e-16 & 268 \\	
\hline
	\end{tabular}

	\end{table}

\subsubsection{Example 3}\label{3.4.3}
Finally, in this section, we compare the algorithm with CG algorithm and AR algorithm to illustrate the effectiveness of our algorithm through the following numerical examples, where the given coefficient matrices are
\begin{equation*}
A =
\begin{pmatrix}
5 & -1 & & & & \\
-1 & 5 & -1 & &  &\\
 & \ddots & \ddots & \ddots &\\
 & & \ddots & \ddots & -1 &\\
 & & & -1& 5& \\
\end{pmatrix}_{n \times n},\; \;
B =
\begin{pmatrix}
6 & 2 & & & & \\
2 & 6 & 2 & &  &\\
 & \ddots & \ddots & \ddots &\\
 & & \ddots & \ddots & 2 &\\
 & & & 2 & 6& \\
\end{pmatrix}_{n \times n},
\end{equation*}

\begin{equation*}
C=I_n.
\end{equation*}
The parameters in AR algorithm is $m=1$, when the matrix order is different, the error data is shown in Table 6.

\renewcommand{\floatpagefraction}{.9}

\begin{table}[!htbp]
	\label{tab6}
	\centering
 	\caption{Numerical results for Example 3}
	\begin{tabular}{ccccc}
	\hline
	algorithm & n & iteration & error & time(s) \\ \hline
	DFP & 128 & 3 & 2.3697e-14 & 0.05 \\
	BFGS & 128 & 3 & 1.2619e-16 & 0.05 \\
	CG & 128 & 15 & 6.2321e-15 & 0.18 \\
	AR & 128 & 3 & 4.1425e-15 & 0.14 \\ \hline
	DFP & 256 & 3 & 2.7295e-14 & 0.21 \\
	BFGS & 256 & 3 & 1.2619e-16 & 0.23 \\
	CG & 256 & 15 & 6.1485e-15 & 0.65 \\
	AR & 256 & 3 & 5.3648e-15 & 0.33 \\ \hline
	DFP & 512 & 3 & 2.7295e-14 & 1.03 \\
	BFGS & 512 & 3 & 1.2619e-16 & 1.15 \\
	CG & 512 & 15 & 6.0531e-15 & 3.32 \\
	AR & 512 & 3 & 7.3523e-15 & 2.03 \\ \hline
	DFP & 1024 & 3 & 2.7327e-14 & 6.05 \\
	BFGS & 1024 & 3 & 1.2619e-16 & 6.37 \\
	CG & 1024 & 15 & 5.9917e-15 & 33 \\
	AR & 1024 & 3 & 1.1720e-14 & 20 \\ \hline
	DFP & 2048 & 3 & 2.7295e-14 & 34 \\
	BFGS & 2048 & 3 & 1.3900e-16 & 36 \\
	CG & 2048 & 15 & 5.9576e-15 & 289 \\
	AR & 2048 & 3 & 1.9263e-14 & 170 \\ \hline
	DFP & 4096 & 3 & 2.7295e-14 & 178 \\
	BFGS & 4096 & 3 & 1.2619e-16 & 224 \\
	CG & 4096 & 15 & 5.9397e-15 & 764 \\
	AR & 4096 & 3 & 1.1815e-14 & 1268 \\
\hline
	\end{tabular}
	\end{table}
	
\subsubsection{Analysis of numerical results}
The numerical results show that Algorithm \ref{alg3} and Algorithm \ref{alg4} can achieve a higher accuracy with a few iterative steps by using Wolfe-Prowell linear search. Furthermore, the CPU time required by these algorithms is relatively less compared to other methods. As the matrix order increases, Algorithm \ref{alg3} and Algorithm \ref{alg4} maintain their computational effectiveness. This verifies that both Algorithm \ref{alg3} and Algorithm \ref{alg4} exhibit excellent computational performance and consume less CPU time than the AR algorithm.

In summary, by comparing with other classical algorithms, BFGS algorithm and DFP algorithm are proven effective for solving the Sylvester matrix equation. Additionally, our two algorithms in this paper require less CPU time compared to the CG algorithm and AR algorithm. Regarding the iteration steps, the CG algorithm requires more steps than other algorithms. In terms of error accuracy, the BFGS algorithm achieves the highest accuracy as the matrix order increases. In conclusion, the effectiveness of our two algorithms in this section for solving the Sylvester matrix equation is confirmed.

\section{ADMM}
In this section, we propose a unified framework for solving the continuous-time algebraic Riccati equation (CARE) by transforming it into an optimization problem and employing the alternating direction multiplier method (ADMM). We provide convergence analysis and numerical results to demonstrate the feasibility and effectiveness of our method. The structure of this section is as follows: First, we review the basic notations and properties of ADMM in Section \ref{sec1}. In Section \ref{sec2}, we present the unified framework of ADMM to solve equation (\ref{eq1.1}) using the ADMM algorithm. Further, in Section \ref{sec3}, we introduce lemmas to prove the convergence of the algorithm, followed by the convergence results of the problem (\ref{eq1.1}). In Section \ref{sec4}, we present numerical results that illustrate the efficiency of our proposed method.

\subsection{Preliminaries}\label{sec1}
In this section, we first introduce the classic ADMM algorithm, review the following basic symbols and properties of ADMM. First, for such a minimization problem
\begin{equation}\label{0.0}
\min_{u \in \mathbb{R}^t}J(u) \; \; s.t. \; Ku =d,
\end{equation}
where $J$ is a real-valued function , $K\in \mathbb{R}^{s \times t}, d\in \mathbb{R}^s$. The augmented Lagrangian function of (\ref{0.0}) is
\begin{equation*}
\mathcal{L}_{\mathcal{A}}(u,\lambda)=J(u)-\left \langle \lambda,Ku-d \right \rangle + \frac{\alpha}{2} \left \| Ku-d \right \|_2^2,
\end{equation*}
where $\lambda \in \mathbb{R}^s$ is a Lagrangian multiplier vector and $\alpha>0$ is a penalty parameter. The Multiplier method iterations are given by
\begin{equation*}
\begin{cases}
u_{k+1}:=\mathop{\arg\min}\limits_{u\in \mathbb{R}^t} \mathcal{L}_{\mathcal{A}}(u,\lambda_k),\\
\lambda_{k+1}:=\lambda_k-\alpha(Ku_{k+1}-d).
\end{cases}
\end{equation*}

Assume $J(u)$ has separable structure, then (\ref{0.0}) can be rewrite as
\begin{equation}\label{1.1}
\min_{x \in \mathbb{R}^n, y \in \mathbb{R}^m}f(x)+g(x) \; \; s.t. \; Ax+By =c,
\end{equation}
where $f$ and $g$ are convex functions, $A\in \mathbb{R}^{s\times n}, B\in \mathbb{R}^{s\times m}$ and $c\in A\in \mathbb{R}^s$. The classical ADMM is an extension of the augmented Lagrangian multiplier method \cite{ref13},\cite{ref14}. There is a vast range of literature in optimization and numerical analysis focused on splitting methods to solve convex programming problems with separable structures, such as (\ref{1.1}). The aim is to devise algorithms with simple, computationally efficient steps. In this paper, we specifically focus on employing ADMM for solving the CARE problem (\ref{1.1}). The augmented Lagrangian function associated with (\ref{1.1}) is
\begin{equation*}
\mathcal{L}_{\mathcal{A}}(x,y,\lambda)=f(x)+g(y)-\left \langle \lambda,Ax+By-c \right \rangle + \frac{\alpha}{2} \left \| Ax+By-c \right \|_2^2,
\end{equation*}
where $\lambda \in \mathbb{R}^s$ is a Lagrangian multiplier vector and $\alpha>0$ is a penalty parameter. The ADMM iterations are given by
\begin{equation*}
\begin{cases}
x_{k+1}:=\mathop{\arg\min}\limits_{x\in \mathbb{R}^n} \mathcal{L}_{\mathcal{A}}(x,y_k,\lambda_k),\\
y_{k+1}:=\mathop{\arg\min}\limits_{y\in \mathbb{R}^m} \mathcal{L}_{\mathcal{A}}(x_{k+1},y,\lambda_k),\\
\lambda_{k+1}:=\lambda_k-\alpha(Ax_{k+1}+By_{k+1}-c).
\end{cases}
\end{equation*}
That is,
\begin{equation*}
\begin{cases}
x_{k+1}:=\mathop{\arg\min}\limits_{x\in \mathbb{R}^n} f(x)-\left \langle \lambda,Ax \right \rangle + \frac{\alpha}{2} \left \| Ax+By_k-c \right \|_2^2,\\
y_{k+1}:=\mathop{\arg\min}\limits_{y\in \mathbb{R}^m} g(y)-\left \langle \lambda,By \right \rangle + \frac{\alpha}{2} \left \| Ax_{k+1}+By-c \right \|_2^2,\\
\lambda_{k+1}:=\lambda_k-\alpha(Ax_{k+1}+By_{k+1}-c).
\end{cases}
\end{equation*}
Which can be interpreted as alternately minimizing the augmented Lagrangian function $\mathcal{L}_{\mathcal{A}}(x,y,\lambda)$ with respect to $x$ then $y$ and then updating the Lagrange multiplier $\lambda$. Note that the asymmetry of the variables $x$ and $y$ updates. If we switch the order, first minimizing over $y$, then over $x$, we obtain a valid but different incarnation of ADMM. Techniques regarding applying ADMM to problems with separable structure are discussed in detail by Bertsekas and Tsitklis (\cite{ref15} Section 3.4.4)
and Glowinski and Frotin \cite{ref16}.

The Lagrangian function of (\ref{1.1}) is
\begin{equation}\label{1.1.1}
\mathcal{L}(x,y,\lambda)=f(x)+g(y)-\left \langle \lambda,Ax+By-c \right \rangle,
\end{equation}
where dual variables $\lambda \in \mathbb{R}^s$ is a Lagrangian multipliers. The dual problem of (\ref{1.1}) is
\begin{equation}\label{1.2}
\max_{\lambda \in \mathbb{R}^s} \inf_{x \in \mathbb{R}^n,y \in \mathbb{R}^m} \mathcal{L}_{\mathcal{A}}(x,y,\lambda).
\end{equation}
As we all know, finding the optimal solution of problem (\ref{1.1}) and problem (\ref{1.2}) is equivalent to finding a saddle point of (\ref{1.1.1}), that $(x^*,y^*)$ is the optimal original solution and $\lambda^*$ is optimal dual solution when the following holds:
\begin{equation}\label{1.3}
\mathcal{L}(x^*,y^*,\lambda) \leqslant \mathcal{L}(x^*,y^*,\lambda^*) \leqslant \mathcal{L}(x,y,\lambda^*),
\end{equation}
that is,
\begin{equation*}
\begin{aligned}
&\max_{\lambda \in \mathbb{R}^s} f(x^*)+g(y^*)-\left \langle \lambda,Ax^*+By^*-c \right \rangle = \mathcal{L}(x^*,y^*,\lambda^*) \\
&= \min_{x \in \mathbb{R}^n,y \in \mathbb{R}^m} f(x)+g(x)-\left \langle \lambda^*,Ax+By-c \right \rangle ,
\end{aligned}
\end{equation*}
then the KKT optimality condition is
\begin{equation}\label{1.4}
\begin{cases}
Ax^*+By^*=c,\\
A^T\lambda^* \in \partial f(x^*),\\
B^T\lambda^* \in \partial g(y^*).
\end{cases}
\end{equation}
Where $\partial$ is subgradient as define:
\begin{equation*}
\partial f(x^*)=\{ p\in \mathbb{R}^n:f(v) \geqslant f(x^*)+ \left \langle  p,v-x^* \right \rangle ,\forall v \in \mathbb{R}^n  \},
\end{equation*}
\begin{equation*}
\partial g(y^*)=\{ q\in \mathbb{R}^m :g(w) \geqslant g(y^*)+ \left \langle  q,w-y^* \right \rangle ,\forall w \in \mathbb{R}^m  \}.
\end{equation*}
For convex optimization problem, if we obtain (\ref{1.3}) with the KKT condition (\ref{1.4}) then $(x^*,y^*,\lambda^*)$ is a saddle point of $\mathcal{L}$.
\subsection{Iterative Method}\label{sec2}
In this section, we convert the continuous-time algebraic Riccati equation (CARE) into an optimization problem. We then propose the use of ADMM to solve equation (\ref{eq1.1}) and provide the iterative scheme of ADMM along with its computational complexity. To enhance the efficiency of ADMM, we also consider an equivalent form of equation (\ref{eq1.1}):
\begin{equation}\label{eq2}
\begin{aligned}
\min_{(X,Y,Z,W)}  \frac{1}{2} & \left \| Y+ZA-WX+K \right \|_F^2   \\
s.t. \; A^TX &=Y , \\
X&=Z , \\
ZN &= W,
\end{aligned}
\end{equation}
where $X,Y,Z,W \in \mathbb{R}^{n \times n}$.

The augmented Lagrangian function of (\ref{eq2}) is
\begin{equation}\label{eq3}
\begin{split}
\mathcal{L}_{\mathcal{A}}(X,Y,Z,W,\Lambda,\Pi,\Gamma)&=\frac{1}{2} \left \| Y+ZA-WX+K \right \|_F^2 -\left \langle \Lambda,A^TX-Y \right \rangle \\
&- \left \langle \Pi,X-Z \right \rangle - \left \langle \Gamma,ZN-W \right \rangle   + \frac{\alpha}{2} \left \| A^TX-Y \right \|_F^2 \\
&+ \frac{\beta}{2} \left \| X-Z \right \|_F^2 + \frac{\gamma}{2} \left \| ZN-W \right \|_F^2,
\end{split}
\end{equation}
where $\Lambda,\Pi,\Gamma \in \mathbb{R}^{n \times n} $ are Lagrangian multipliers and $\alpha,\beta,\gamma>0$ are penalty parameters. The ADMM iterations for (\ref{eq3}) are given by
\begin{equation*}
\begin{cases}
X_{k+1}:=\mathop{\arg\min}\limits_{X\in \mathbb{R}^{n \times n}} \mathcal{L}_{\mathcal{A}}(X,Y_k,Z_k,W_k,\Lambda_k,\Pi_k,\Gamma_k),\\
Y_{k+1}:=\mathop{\arg\min}\limits_{Y\in \mathbb{R}^{n \times n}} \mathcal{L}_{\mathcal{A}}(X_{k+1},Y,Z_k,W_k,\Lambda_k,\Pi_k,\Gamma_k),\\
Z_{k+1}:=\mathop{\arg\min}\limits_{Z\in \mathbb{R}^{n \times n}} \mathcal{L}_{\mathcal{A}}(X_{k+1},Y_{k+1},Z,W_k,\Lambda_k,\Pi_k,\Gamma_k),\\
W_{k+1}:=\mathop{\arg\min}\limits_{Z\in \mathbb{R}^{n \times n}} \mathcal{L}_{\mathcal{A}}(X_{k+1},Y_{k+1},Z_{k+1},W,\Lambda_k,\Pi_k,\Gamma_k),\\
\Lambda_{k+1}:=\Lambda_k-\alpha(A^TX_{k+1}-Y_{k+1}),\\
\Pi_{k+1}:=\Pi_k-\beta(X_{k+1}-Z_{k+1}),\\
\Gamma_{k+1}:=\Gamma_k-\gamma(Z_{k+1}N-W_{k+1}).
\end{cases}
\end{equation*}
Thus, these steps can be written in closed form as
\begin{equation}\label{eq4}
\begin{cases}
X_{k+1}=[W_k^TW_k+\alpha AA^T+\beta I_n]^{-1}[W_k^T(Y_k+Z_kA+K)+A\Lambda_k+\Pi_k+\alpha AY_k+\beta Z_k],\\
Y_{k+1}=(1+\alpha)^{-1}[W_kX_{k+1}+\alpha A^TX_{k+1}-Z_kA-K-\Lambda_k],\\
Z_{k+1}=[(W_kX_{k+1}-Y_{k+1}-K)A^T-\Pi_k+\Gamma_kN^T+\beta X_{k+1}+\gamma W_kN^T][AA^T+\beta I_n+\gamma NN^T]^{-1},\\
W_{k+1}=[(Y_{k+1}+Z_{k+1}A+K)X_{k+1}^T-\Gamma_k+\gamma Z_{k+1}N][X_{k+1}X_{k+1}^T+\gamma I_n]^{-1},\\
\Lambda_{k+1}=\Lambda_k-\alpha(A^TX_{k+1}-Y_{k+1}),\\
\Pi_{k+1}=\Pi_k-\beta(X_{k+1}-Z_{k+1}),\\
\Gamma_{k+1}=\Gamma_k-\gamma(Z_{k+1}N-W_{k+1}).
\end{cases}
\end{equation}

A point $(X^*,Y^*,Z^*,W^*)$ satisfies the KKT optimality conditions for the (\ref{eq2}) if there exist $\Lambda^*,\Pi^*,\Gamma^*$ such that
\begin{equation}\label{eq5.1}
\begin{cases}
W^{*T}W^*X^*-W^{*T}(Y^*+Z^*A+K)-A\Lambda^*-\Pi^*=0,\\
Y^*+Z^*A-W^*X^*+K+\Lambda^*=0,\\
Z^*AA^T+(Y^*-W^*X^*+K)A^T+\Pi^*-\Gamma^*N^T=0,\\
W^*X^*X^{*T}-(Y^*+Z^*A+K)X^{*T}+\Gamma^*=0,\\
A^TX^*-Y^*=0,\\
X^*-Z^*=0,\\
Z^*N-W^*=0.
\end{cases}
\end{equation}
The ADMM algorithm for solving optimization problem (\ref{eq2}) is summarized in Algorithm \ref{alg5}.
\begin{algorithm}[H]
\caption{ADMM}
\label{alg5}
\renewcommand{\algorithmicrequire}{\textbf{Input:}}
\renewcommand{\algorithmicensure}{\textbf{Output:}}
\begin{algorithmic}
\REQUIRE $X_0, Y_0, Z_0, W_0, \Lambda_0, \Pi_0, \Gamma_0, \epsilon>0, \alpha, \beta, \gamma, kmaxiter>0$, set $k \leftarrow 1$.

\STATE  1.When $k \leqslant kmaxiter $, calculate $X_k,Y_k,Z_k,W_k,\Lambda_k,\Pi_k,\Gamma_k$ from (\ref{eq4}).

\STATE 2.If $\left \| A^TX+XA-XNX+K \right \|_F \leqslant \epsilon$, stop iteration.

\STATE 3.Set $k \leftarrow k+1$ and go to step 1.

\ENSURE $X_k$.
\end{algorithmic}
\end{algorithm}

We analyze the computational complexity in each iteration of Algorithm \ref{alg5}, it's mainly controlled by matrix multiplication and matrix inverse operation from the iterative formula (\ref{eq4}). There are three inverse operations in each iteration, so we use LU decomposition to reduce the computational complexity to $\frac{1}{3}n^3$, the iterative formula of $X$ needs $14n^3+3n^2$ computational complexity, the iterative formula of $Y$ needs $4n^3+3n^2$ computational complexity, the iterative formula of $Z$ needs $10n^3+7n^2$ computational complexity, the iterative formula of $W$ needs $11n^3+2n^2$ computational complexity, the iterative formula of $\Lambda$ ,$\Pi$ and $\Gamma$ all need $3n^2$ computational complexity. Therefore, the total computational complexity is $39n^3+24n^2$ in Algorithm \ref{alg5}.

\subsection{Convergent Analysis}\label{sec3}
Now, we recall some results which will be used in following.
\begin{lemma}\label{lem11}
For any $X,Y\in \mathbb{R}^{m \times n}$ and $t\in[0,1]$, then we have
\begin{equation*}
\left \| X-Y \right \|_F^2 = \left \| X \right \|_F^2 - 2\left \langle X,Y \right \rangle + \left \| Y \right \|_F^2,
\end{equation*}
and
\begin{equation*}
\left \| tX+(1-t)Y \right \|_F^2 = t\left \| X \right \|_F^2 + (1-t)\left \| Y \right \|_F^2 - t(1-t)\left \| X-Y \right \|_F^2.
\end{equation*}
In particular,
\begin{equation*}
\left \| tX+(1-t)Y \right \|_F^2 \leqslant t\left \| X \right \|_F^2 + (1-t)\left \| Y \right \|_F^2 .
\end{equation*}
\end{lemma}

\begin{lemma}\label{lem12}
Let $\mathbb{D}$ be a convex set in a real vector space and let $f:\mathbb{D} \subseteq \mathbb{R}^n \to \mathbb{R}$ be a   differentiable function, then the following results are equivalent:

(1) $f$ is called strongly convex with parameter $\mu >0$;

(2) the following inequality holds for all $x,y\in \mathbb{D}$ and $t\in[0,1]$:
\begin{equation*}
f(tx+(1-t)y) \leqslant tf(x)+(1-t)f(y)-\frac{1}{2}\mu t(1-t)\left \| x-y \right \|_2^2 ,
\end{equation*}

(3) the following inequality holds for all $x,y\in \mathbb{D}$:
\begin{equation*}
f(x)-f(y) \geqslant \nabla f(x)^T(x-y)+\frac{1}{2}\mu \left \| x-y \right \|_2^2 ,
\end{equation*}

(4) the following inequality holds for all $x,y\in \mathbb{D}$:
\begin{equation*}
\left \langle \nabla f(x)-\nabla f(y),x-y \right \rangle \geqslant \mu \left \| x-y \right \|_2^2 .
\end{equation*}
Moreover, if function $f$ is twice continuously differentiable, then $f$ is strongly convex with parameter $\mu>0$ if and only if $\nabla^2 f(x)-\mu I$ is symmetric positive semi-definite for all $x\in \mathbb{D}$, where $\nabla^2 f$ is the Hessian matrix.
\end{lemma}

For $A\in \mathbb{R}^{s\times m},X\in \mathbb{R}^{m\times n}$ and $B\in \mathbb{R}^{n\times t}$, based on the Kronecker product and inner product, we have
\begin{equation*}
\left \| AXB \right \|_F^2 = \left \| (B^T \otimes A)vec(X) \right \|_2^2,
\end{equation*}
and
\begin{equation*}
\left \langle vec(X),vec(Y) \right \rangle = \left \langle X,Y \right \rangle.
\end{equation*}
Where $\otimes$ senote the Kronecker product, i.e., $A\otimes B=(a_{ij}B)$ and
\begin{equation*}
vec(X)=(x_{11},x_{21},...,x_{m1},x_{12},x_{22},...,x_{m2},...,x_{1n},x_{2n},...,x_{mn})^T \in \mathbb{R}^{mn}.
\end{equation*}

Thus, we get the following lemma.
\begin{lemma}\label{lem13}
Given the matrices $M\in \mathbb{R}^{s\times m}$ and $N\in \mathbb{R}^{s\times n}$, if $M$ has full column rank, then the function
\begin{equation*}
\mathcal{F}(X)=\frac{1}{2} \left \| MX-N \right \|_F^2
\end{equation*}
is strongly convex for all $X\in \mathbb{R}^{m\times n}$ and the parameter $\mu=\lambda_{min}(M^TM)>0$, i.e.. The following inequality holds for all $X,Y\in \mathbb{R}^{m\times n}$:
\begin{equation}\label{3.3.1}
\mathcal{F}(X)-\mathcal{F}(Y) \geqslant \left \langle \nabla \mathcal{F}(Y),X-Y \right \rangle + \frac{\mu}{2} \left \| X-Y \right \|_F^2.
\end{equation}
\end{lemma}
\begin{proof}
For the function
\begin{equation*}
f(vec(X))=\frac{1}{2} \left \| (I_n \otimes M)vec(X)-vec(N) \right \|_2^2 ,
\end{equation*}
we have
\begin{equation*}
\nabla^2 f(vec(X))=(I_n \otimes M)^T(I_n \otimes M)=I_n\otimes (M^TM),
\end{equation*}
thus
\begin{equation*}
\begin{aligned}
\nabla^2 f(vec(X))-\lambda_{min}(M^TM)I_{nm}
&=I_n \otimes (M^TM)-\lambda_{min}(M^TM)I_{nm}\\
&=I_n \otimes (M^TM)-\lambda_{min}(M^TM)(I_{n}\otimes I_m)\\
&=I_n \otimes (M^TM-\lambda_{min}(M^TM)I_m).
\end{aligned}
\end{equation*}
Since $M$ has full column rank, then $M^TM$ is symmetric positive definite. Thus, $M^TM-\lambda_{min}(M^TM)I_m$ is symmetric positive semi-definite. Based on the property of the Kronecker product, it is easy to know that $I_n \otimes (M^TM-\lambda_{min}(M^TM)I_m)$ is symmetric positive semi-definite, i.e., $\nabla^2 f(vec(X))-\lambda_{min}(M^TM)I_{nm}$ is symmetric positive semi-definite. Due to Lemma \ref{lem12}, we can get that $f$ is strongly convex with parameter $\mu=\lambda_{min}(M^TM)>0$. According to the equivalency of the functions $f$ and $\mathcal{F}$, we obtain that the function $\mathcal{F}$ is strongly convex and it is easy to get the inequality (\ref{3.3.1}).
\end{proof}

To the best of our knowledge, there is no global convergence result in general for nonconvex programs or convex programs with three or more blocks. Note that there are four blocks in our iterative formula. Due to these difficulties, we provide a convergence property of the proposed ADMM that holds only under some assumptions.
\begin{theorem}\label{thm4}
Let ${(X_k,Y_k,Z_k,W_k,\Lambda_k,\Pi_k,\Gamma_k)}$ be a sequence generated by ADMM(\ref{eq4}). If the multiplier sequence $(\Lambda_k,\Pi_k,\Gamma_k)$ is bounded and satisfies
\begin{equation}\label{eq5}
\sum_{k=0}^\infty (\left \| \Lambda_{k+1}-\Lambda_k \right \|_F^2 + \left \| \Pi_{k+1}-\Pi_k \right \|_F^2 + \left \| \Gamma_{k+1}-\Gamma_k \right \|_F^2)<\infty.
\end{equation}
Then any accumulation point of ${(X_k,Y_k,Z_k,W_k,\Lambda_k,\Pi_k,\Gamma_k)}$ satisfies the KKT optimality condition (\ref{eq5.1}) of problem (\ref{eq2}).
\end{theorem}
\begin{proof}
First, we claim
\begin{equation*}
\begin{aligned}
\begin{split}
& \mathcal{L}_{\mathcal{A}}(X_k,Y_k,Z_k,W_k,\Lambda_k,\Pi_k,\Gamma_k)-\mathcal{L}_{\mathcal{A}}(X_{k+1},Y_{k+1},Z_{k+1},W_{k+1},\Lambda_{k+1},\Pi_{k+1},\Gamma_{k+1}) \\
& \geqslant \frac{\beta}{2}\left \| X_k - X_{k+1} \right \|_F^2 + \frac{\alpha}{2}\left \| Y_k - Y_{k+1} \right \|_F^2 + \frac{\beta}{2}\left \| Z_k - Z_{k+1} \right \|_F^2 + \frac{\gamma}{2}\left \| W_k - W_{k+1} \right \|_F^2 \\
& -\frac{1}{\alpha}\left \| \Lambda_k - \Lambda_{k+1} \right \|_F^2 -\frac{1}{\beta}\left \| \Pi_k - \Pi_{k+1} \right \|_F^2 -\frac{1}{\gamma}\left \| \Gamma_k - \Gamma_{k+1} \right \|_F^2 .
\end{split}
\end{aligned}
\end{equation*}
Since the augmented Lagrangian function of (\ref{eq2}) can be rewritten as
\begin{equation}\label{eq6}
\begin{aligned}
\begin{split}
\mathcal{L}_{\mathcal{A}}(X,Y,Z,W,\Lambda,\Pi,\Gamma)
&=\frac{1}{2} \left \| Y+ZA-WX+K \right \|_F^2 +\frac{\alpha}{2} \left \|(A^TX-Y)- \frac{\Lambda}{\alpha} \right \|_F^2 - \frac{1}{2\alpha} \left \| \Lambda \right \|_F^2 \\
&+ \frac{\beta}{2} \left \|(X-Z) - \frac{\Pi}{\beta} \right \|_F^2 - \frac{1}{2\beta}\left \| \Pi \right \|_F^2 + \frac{\gamma}{2} \left \|(ZN-W)- \frac{\Gamma}{\gamma} \right \|_F^2 - \frac{1}{2\gamma}\left \| \Gamma \right \|_F^2.
\end{split}
\end{aligned}
\end{equation}
From (\ref{eq6}), Lemmas \ref{lem11} and \ref{lem13}, it is easy to known that the augmented Lagrangian function $\mathcal{L}_{\mathcal{A}}$ is strongly convex with respect to each variable of $X$, $Y$, $Z$ and $W$, respectively. For $X$, by Lemmas \ref{lem11} and \ref{lem13} and the identity matrix has full column rank, it holds that
\begin{equation}\label{eq8}
\begin{aligned}
&\mathcal{L}_{\mathcal{A}}(X+\triangle X,Y,Z,W,\Lambda,\Pi,\Gamma)-\mathcal{L}_{\mathcal{A}}(X,Y,Z,W,\Lambda,\Pi,\Gamma) \\
&\geqslant \left \langle \partial_X \mathcal{L}_{\mathcal{A}}(X,Y,Z,W,\Lambda,\Pi,\Gamma),\triangle X \right \rangle + \frac{\beta}{2}\left \| \triangle X \right \|_F^2,
\end{aligned}
\end{equation}
for any $X$ and $\triangle X$, and $X$ being a minimizer of $\mathcal{L}_{\mathcal{A}}(X,Y,Z,W,\Lambda,\Pi,\Gamma)$, so that
\begin{equation}\label{eq9}
 \left \langle \partial_X \mathcal{L}_{\mathcal{A}}(X,Y,Z,W,\Lambda,\Pi,\Gamma),\triangle X \right \rangle \geqslant 0.
\end{equation}
Combining (\ref{eq8}), (\ref{eq9}) and $X_{k+1}:=\mathop{\arg\min}\limits_{X\in \mathbb{R}^{n \times n}} \mathcal{L}_{\mathcal{A}}(X,Y_k,Z_k,W_k,\Lambda_k,\Pi_k,\Gamma_k)$, we have
\begin{equation}\label{eq10}
\mathcal{L}_{\mathcal{A}}(X_k,Y_k,Z_k,W_k,\Lambda_k,\Pi_k,\Gamma_k)-\mathcal{L}_{\mathcal{A}}(X_{k+1},Y_k,Z_k,W_k,\Lambda_k,\Pi_k,\Gamma_k) \geqslant \frac{\beta}{2}\left \| X_k - X_{k+1} \right \|_F^2.
\end{equation}
Similarly, we have
\begin{equation}\label{eq11}
\mathcal{L}_{\mathcal{A}}(X_{k+1},Y_k,Z_k,W_k,\Lambda_k,\Pi_k,\Gamma_k)-\mathcal{L}_{\mathcal{A}}(X_{k+1},Y_{k+1},Z_k,W_k,\Lambda_k,\Pi_k,\Gamma_k) \geqslant \frac{\alpha}{2}\left \| Y_k - Y_{k+1} \right \|_F^2,
\end{equation}
\begin{equation}\label{eq12}
\mathcal{L}_{\mathcal{A}}(X_{k+1},Y_{k+1},Z_k,W_k,\Lambda_k,\Pi_k,\Gamma_k)-\mathcal{L}_{\mathcal{A}}(X_{k+1},Y_{k+1},Z_{k+1},W_k,\Lambda_k,\Pi_k,\Gamma_k) \geqslant \frac{\beta}{2}\left \| Z_k - Z_{k+1} \right \|_F^2,
\end{equation}
\begin{equation}\label{eq13}
\mathcal{L}_{\mathcal{A}}(X_{k+1},Y_{k+1},Z_{k+1},W_k,\Lambda_k,\Pi_k,\Gamma_k)-\mathcal{L}_{\mathcal{A}}(X_{k+1},Y_{k+1},Z_{k+1},W_{k+1},\Lambda_k,\Pi_k,\Gamma_k) \geqslant \frac{\gamma}{2}\left \| W_k - W_{k+1} \right \|_F^2.
\end{equation}
Due to $\Lambda_{k+1} = \Lambda_k-\alpha(A^TX_{k+1}-Y_{k+1})$, we get
\begin{equation}\label{eq14}
\begin{aligned}
\begin{split}
&\mathcal{L}_{\mathcal{A}}(X_{k+1},Y_{k+1},Z_{k+1},W_{k+1},\Lambda_k,\Pi_k,\Gamma_k)-\mathcal{L}_{\mathcal{A}}(X_{k+1},Y_{k+1},Z_{k+1},W_{k+1},\Lambda_{k+1},\Pi_k,\Gamma_k)\\
&= -\left \langle  \Lambda_k,A^TX_{k+1}-Y_{k+1} \right \rangle + \left \langle  \Lambda_{k+1},A^TX_{k+1}-Y_{k+1} \right \rangle\\
& = \left \langle \Lambda_{k+1}-\Lambda_k,\frac{\Lambda_k-\Lambda_{k+1}}{\alpha} \right \rangle\\
&= -\frac{1}{\alpha} \left \| \Lambda_k - \Lambda_{k+1} \right \|_F^2.
\end{split}
\end{aligned}
\end{equation}
Similarly, we have
\begin{equation}\label{eq15}
\begin{aligned}
\mathcal{L}_{\mathcal{A}}&(X_{k+1},Y_{k+1},Z_{k+1},W_{k+1},\Lambda_{k+1},\Pi_k,\Gamma_k)-\mathcal{L}_{\mathcal{A}}(X_{k+1},Y_{k+1},Z_{k+1},W_{k+1},\Lambda_{k+1},\Pi_{k+1},\Gamma_k) \\
&= -\frac{1}{\beta} \left \| \Pi_k - \Pi_{k+1} \right \|_F^2,
\end{aligned}
\end{equation}
\begin{equation}\label{eq16}
\begin{aligned}
&\mathcal{L}_{\mathcal{A}}(X_{k+1},Y_{k+1},Z_{k+1},W_{k+1},\Lambda_{k+1},\Pi_{k+1},\Gamma_k)-\mathcal{L}_{\mathcal{A}}(X_{k+1},Y_{k+1},Z_{k+1},W_{k+1},\Lambda_{k+1},\Pi_{k+1},\Gamma_{k+1}) \\
&= -\frac{1}{\gamma} \left \| \Gamma_k - \Gamma_{k+1} \right \|_F^2.
\end{aligned}
\end{equation}
Taking summation of (\ref{eq10})-(\ref{eq16}), we can obtain
\begin{equation}\label{eq17}
\begin{aligned}
\begin{split}
& \mathcal{L}_{\mathcal{A}}(X_k,Y_k,Z_k,W_k,\Lambda_k,\Pi_k,\Gamma_k)-\mathcal{L}_{\mathcal{A}}(X_{k+1},Y_{k+1},Z_{k+1},W_{k+1},\Lambda_{k+1},\Pi_{k+1},\Gamma_{k+1}) \\
& \geqslant \frac{\beta}{2}\left \| X_k - X_{k+1} \right \|_F^2 + \frac{\alpha}{2}\left \| Y_k - Y_{k+1} \right \|_F^2 + \frac{\beta}{2}\left \| Z_k - Z_{k+1} \right \|_F^2 + \frac{\gamma}{2}\left \| W_k - W_{k+1} \right \|_F^2 \\
& -\frac{1}{\alpha}\left \| \Lambda_k - \Lambda_{k+1} \right \|_F^2 -\frac{1}{\beta}\left \| \Pi_k - \Pi_{k+1} \right \|_F^2 -\frac{1}{\gamma}\left \| \Gamma_k - \Gamma_{k+1} \right \|_F^2 .
\end{split}
\end{aligned}
\end{equation}

Now, we show $(X_{k+1},Y_{k+1},Z_{k+1},W_{k+1},\Lambda_{k+1},\Pi_{k+1},\Gamma_{k+1})-(X_k,Y_k,Z_k,W_k,\Lambda_k,\Pi_k,\Gamma_k) \to 0$, due to (\ref{eq5}) so that $(\Lambda_{k+1},\Pi_{k+1},\Gamma_{k+1})-(\Lambda_k,\Pi_k,\Gamma_k) \to 0$, so we only need to prove that $(X_{k+1},Y_{k+1},Z_{k+1},W_{k+1})-(X_k,Y_k,Z_k,W_k)\to 0$, From the boundedness of ${(\Lambda_k,\Pi_k,\Gamma_k)}$ and (\ref{eq6}), we get that $\mathcal{L}_{\mathcal{A}}(X,Y,Z,W,\Lambda_k,\Pi_k,\Gamma_k)$ is bounded below for any $X,Y,Z,W,k$. Taking summation of the above inequality (\ref{eq17}) and note that $\mathcal{L}_{\mathcal{A}}(X,Y,Z,W,\Lambda_k,\Pi_k,\Gamma_k)$ is bounded belw, we have
\begin{equation*}
\begin{split}
\sum_{k=0}^\infty (\frac{\beta}{2}\left \| X_k - X_{k+1} \right \|_F^2 + \frac{\alpha}{2}\left \| Y_k - Y_{k+1} \right \|_F^2 + \frac{\beta}{2}\left \| Z_k - Z_{k+1} \right \|_F^2 + \frac{\gamma}{2}\left \| W_k - W_{k+1} \right \|_F^2 ) \\- \sum_{k=0}^\infty (\frac{1}{\alpha}\left \| \Lambda_k - \Lambda_{k+1} \right \|_F^2 +\frac{1}{\beta}\left \| \Pi_k - \Pi_{k+1} \right \|_F^2 +\frac{1}{\gamma}\left \| \Gamma_k - \Gamma_{k+1} \right \|_F^2 ) < \infty.
\end{split}
\end{equation*}
Since $\sum_{k=0}^\infty (\left \| \Lambda_k - \Lambda_{k+1} \right \|_F^2 +\left \| \Pi_k - \Pi_{k+1} \right \|_F^2 +\left \| \Gamma_k - \Gamma_{k+1} \right \|_F^2 ) < \infty$, it holds that
\begin{equation*}
\sum_{k=0}^\infty (\frac{\beta}{2}\left \| X_k - X_{k+1} \right \|_F^2 + \frac{\alpha}{2}\left \| Y_k - Y_{k+1} \right \|_F^2 + \frac{\beta}{2}\left \| Z_k - Z_{k+1} \right \|_F^2 + \frac{\gamma}{2}\left \| W_k - W_{k+1} \right \|_F^2 ) < \infty.
\end{equation*}
Thus
\begin{equation*}
(X_{k+1},Y_{k+1},Z_{k+1},W_{k+1})-(X_k,Y_k,Z_k,W_k)\to 0.
\end{equation*}
Finally, we are ready to prove the result of this theorem. From (\ref{eq4}), we have
\begin{equation}\label{eq18}
\begin{aligned}
&[W_k^TW_k+\alpha AA^T+\beta I_n](X_{k+1}-X_k)\\
&=W_k^T(Y_k+Z_kA+K)+A\Lambda_k+\Pi_k+\alpha AY_k+\beta Z_k-W_k^TW_kX_k-\alpha AA^TX_k-\beta X_k\\
&=W_k^T(Y_k+Z_kA+K)+A\Lambda_k+\Pi_k-W_k^TW_kX_k+\alpha A(Y_k-A^TX_k)+\beta(Z_k-X_k),
\end{aligned}
\end{equation}
\begin{equation}\label{eq19}
\begin{aligned}
&(1+\alpha)(Y_{k+1}-Y_k)\\
&=W_kX_{k+1}+\alpha A^TX_{k+1}-Z_kA-K-\Lambda_k-Y_k-\alpha Y_k\\
&=W_kX_{k+1}+\alpha (A^TX_{k+1}-Y_{k+1})-Z_kA-K-\Lambda_k-Y_k+\alpha(Y_{k+1}-Y_k),
\end{aligned}
\end{equation}
\begin{equation}\label{eq20}
\begin{aligned}
&(Z_{k+1}-Z_k)[AA^T+\beta I_n+\gamma NN^T]\\
&=(W_kX_{k+1}-Y_{k+1}-K)A^T-\Pi_k+\Gamma_kN^T+\beta X_{k+1}\\
&+\gamma W_kN^T-Z_kAA^T-\beta Z_k-\gamma Z_kNN^T\\
&=\beta(X_{k+1}-Z_{k+1})+\beta(Z_{k+1}-Z_k)+\gamma(W_k- Z_kN)N^T \\
&+ (W_kX_{k+1}-Y_{k+1}-K)A^T-\Pi_k+\Gamma_kN^T-Z_kAA^T,
\end{aligned}
\end{equation}

\begin{equation}\label{eq21}
\begin{aligned}
&(W_{k+1}-W_k)[X_{k+1}X_{k+1}^T+\gamma I_n]\\
&=(Y_{k+1}+Z_{k+1}A+K)X_{k+1}^T-\Gamma_k+\gamma Z_{k+1}N - W_kX_{k+1}X_{k+1}^T-\gamma W_k\\
&=\gamma(Z_{k+1}N - W_{k+1})+\gamma(W_{k+1} - W_k)+(Y_{k+1}+Z_{k+1}A+K)X_{k+1}^T-\Gamma_k- W_kX_{k+1}X_{k+1}^T,
\end{aligned}
\end{equation}

\begin{equation}\label{eq22}
\Lambda_{k+1}-\Lambda_k=-\alpha(A^TX_{k+1}-Y_{k+1}),
\end{equation}

\begin{equation}\label{eq23}
\Pi_{k+1}-\Pi_k=-\beta(X_{k+1}-Z_{k+1}),
\end{equation}

\begin{equation}\label{eq24}
\Gamma_{k+1}-\Gamma_k=-\gamma(Z_{k+1}N-W_{k+1}).
\end{equation}
Since $(\Lambda_k,\Pi_k,\Gamma_k)$ is bounded, then ${(X_k,Y_k,Z_k,W_k)}$ is bounded, and $(X_{k+1},Y_{k+1},Z_{k+1},W_{k+1},\Lambda_{k+1},\Pi_{k+1},\Gamma_{k+1})-(X_k,Y_k,Z_k,W_k,\Lambda_k,\Pi_k,\Gamma_k) \to 0$, so let the left and right hand sides in (\ref{eq18})-(\ref{eq24}) all go to zero, then
\begin{equation}\label{eq25}
\begin{aligned}
&W_k^T(Y_k+Z_kA+K)+A\Lambda_k+\Pi_k-W_k^TW_kX_k \to 0,\\
&W_kX_{k+1}+\alpha (A^TX_{k+1}-Y_{k+1})-Z_kA-K-\Lambda_k-Y_k \to 0,\\
&(W_kX_{k+1}-Y_{k+1}-K)A^T-\Pi_k+\Gamma_kN^T-Z_kAA^T \to 0,\\
&(Y_{k+1}+Z_{k+1}A+K)X_{k+1}^T-\Gamma_k- W_kX_{k+1}X_{k+1}^T \to 0,\\
&A^TX_{k+1}-Y_{k+1}\to 0,\\
&X_{k+1}-Z_{k+1}\to 0,\\
&Z_{k+1}N-W_{k+1}.
\end{aligned}
\end{equation}

For any limit point ${(X_k,Y_k,Z_k,W_k)}$ of the sequence $(X^*,Y^*,Z^*,W^*)$, there exists a subsequence ${X_{k_i},Y_{k_i},Z_{k_i},W_{k_i}}$ converging to $(X^*,Y^*,Z^*,W^*)$. The boundednedd of $(\Lambda_k,\Pi_k,\Gamma_k)$ implies the existence of a sub-subsequence ${(\Lambda_{k_{i_j}},\Pi_{k_{i_j}})}$ of ${(\Lambda_{k_i},\Pi_{k_i},\Gamma_{k_i})}$ converging to some point $(\Lambda^*,\Pi^*,\Gamma^*)$. Hence, $(X^*,Y^*,Z^*,W^*,\Lambda^*,\Pi^*,\Gamma^*)$ is a limit point of ${(X_k,Y_k,Z_k,W_k,\Lambda_k,\Pi_k,\Gamma_k)}$. Taking limitation of (\ref{eq25}), it following that
\begin{equation*}
\begin{cases}
W^{*T}W^*X^*-W^{*T}(Y^*+Z^*A+K)-A\Lambda^*-\Pi^*=0,\\
Y^*+Z^*A-W^*X^*+K+\Lambda^*=0,\\
Z^*AA^T+(Y^*-W^*X^*+K)A^T+\Pi^*-\Gamma^*N^T=0,\\
W^*X^*X^{*T}-(Y^*+Z^*A+K)X^{*T}+\Gamma^*=0,\\
A^TX^*-Y^*=0,\\
X^*-Z^*=0,\\
Z^*N-W^*=0.
\end{cases}
\end{equation*}
This completes the proof.

\end{proof}

From the proof of Theorem \ref{thm4}, we can get the following corollary.
\begin{corollary}
Let ${(X_k,Y_k,Z_k,W_k,\Lambda_k,\Pi_k,\Gamma_k)}$ be a sequence generated by ADMM (\ref{eq4}). Whenever the sequence converges, the limit satisfies the KKT optimality condition (\ref{eq5.1}) of the problem (\ref{eq3}).
\end{corollary}

\subsection{Experimental Results}\label{sec4}
In this section, we offer two cooresponding numerical examples to illustrate the efficiency of Algorithm \ref{alg5}. Further, we compare with Newton method for solving the problem (\ref{eq3}), and come to a conclusion that Algorithm \ref{alg5} are less efficient for solving CARE. All code is written in Matlab language. Denote iteration and error by the iteration step and error of the objective function. We take the matrix order "$n$" as 16, 32, 64, 128, 256, 512, 1024, 2048 and 4096 in Example 2, initialize $X_0$ as an null matrix, and error precision is $\epsilon=10^{-8}$. We use the error of matrix equation as
\begin{equation*}
\left \| R_k \right \|_F=\left \| A^TX+XA-XNX+K \right \|_F.
\end{equation*}
\subsubsection{Example 1}
This example is a 9-order continuous-time state space model of a tubular ammonia reactor, and its coefficient matrices are
\begin{equation*}
A =
\begin{bmatrix}
-4.019& 5.12 & 0 & 0 & -2.082 & 0 & 0 & 0 & 0.87\\
-0.346& 0.986 & 0 & 0 & -2.34 & 0 & 0 & 0 & 0.97 \\
-7.909& 15.407 & -4.096 & 0 & -6.45 & 0 & 0 & 0 & 2.68 \\
-21.816& 35.606 & -0.339 & -3.87 & -17.8 & 0 & 0 & 0 & 7.39 \\
-60.196& 98.188 & -7.907 & 0.34 & -53.008 & 0 & 0 & 0 & 20.4 \\
0& 0 & 0 & 0 & 94.0 & -147.2 & 0 & 53.2 & 0 \\
0& 0 & 0 & 0 & 0 & 94.0 & -147.2 & 0 & 53.2  \\
0& 0 & 0 & 0 & 0 & 12.8 & 0 & -31.6 & 0   \\
0& 0 & 0 & 0 & 12.8 & 0 & 0 & 18.8 & -31.6    \\
\end{bmatrix},
\end{equation*}
\begin{equation*}
B^T =
\begin{bmatrix}
0.010& 0.003 & 0.009 & 0.024 & 0.068 & 0 & 0 & 0 & 0\\
-0.011& 0.021 & -0.059 & -0.162 & -0.445 & 0 & 0 & 0 & 0 \\
-0.151& 0 & 0 & 0 & 0 & 0 & 0 & 0 & 0 \\
\end{bmatrix},
\end{equation*}

\begin{equation*}
N=BB^T, K=I_9.
\end{equation*}

For the selection of parameters, different methods have different effects in this example. We will compare three parameter selection methods as $\alpha=0.0465,\beta=63.51,\gamma=0.0428$, $\alpha=0.2,\beta=100,\gamma=0.01$ and $\alpha=0.2,\beta=100,\gamma=0.1$, the error data is shown in Table 7.
\begin{table}[!htbp]
	\label{tab7}
	\centering
 \caption{Numerical results for Example 1}
	\begin{tabular}{cccccc}
	\hline
	$\alpha$ & $\beta$ & $\gamma$ & iteration & error & time(s) \\ \hline
	0.0465 & 63.51 & 0.0428 & 6715 & 9.9961e-09 & 0.2292\\
	0.2 & 100 & 0.01 &17869 & 8.5193e-09 & 1.0675 \\
	0.2 & 100 & 0.1 & 12329 & 9.9939e-09 & 0.6784 \\ \hline
	\end{tabular}

	\end{table}

\subsubsection{Example 2}
We compare the algorithm with Newton algorithm, then our algorithm should be improved through the following numerical examples, where the given coefficient matrices are:
\begin{equation*}
A =
\begin{pmatrix}
6 & 1 & & & & \\
2 & 6 & 1 & &  &\\
 & \ddots & \ddots & \ddots &\\
 & & \ddots & \ddots & 1 &\\
 & & & 2& 6& \\
\end{pmatrix}_{n \times n},
B^T =
\begin{pmatrix}
5 & 1 & & & & \\
2 & 5 & 1 & &  &\\
 & \ddots & \ddots & \ddots &\\
 & & \ddots & \ddots & 1 &\\
 & & & 2 & 5& \\
\end{pmatrix}_{n \times n},
\end{equation*}

\begin{equation*}
N=BB^T, K=I_n.
\end{equation*}

When we set the parameter are $\alpha =0.91 , \beta= 2.8 , \gamma= 0.0014$ that get the best calculation effect. Then when the matrix order is different, the error data is shown in Table 8.
\begin{table}[!htbp]
	\label{tab8}
	\centering
 \caption{Numerical results for Example 2}
	\begin{tabular}{ccccc}
	\hline
	algorithm & n & iteration & error & time(s) \\ \hline
	ADMM & 16 & 563 & 9.9673e-09 & 0.05 \\
	Newton & 16 & 83 & 6.0989e-08 & 0.09 \\\hline
	ADMM & 32 & 602 & 9.9315e-09 & 0.15 \\
	Newton & 32 & 76 & 6.2125e-08 & 0.11 \\\hline
	ADMM & 64 & 627 & 9.4188e-09 & 0.38 \\
	Newton & 64 & 76 & 6.5353e-08 & 0.52 \\\hline
	ADMM & 128 & 641 & 9.8931e-09 & 1.60 \\
	Newton & 128 & 76 & 6.2467e-08 & 1.38 \\\hline
	ADMM & 256 & 661 & 9.8646e-09 & 6.67 \\
	Newton & 256 & 76 & 7.0731e-08 & 4.97 \\\hline
	ADMM & 512 & 674 & 9.9501e-09 & 38 \\
	Newton & 512 & 76 & 7.5188e-08 & 23 \\\hline
	ADMM & 1024 & 687 & 9.9190e-09 & 227 \\
	Newton & 1024 & 76 & 7.8919e-08 & 104 \\\hline
	ADMM & 2048 & 700 & 9.8274e-09 & 2129 \\
	Newton & 2048 & 76 & 8.3520e-08 & 983 \\\hline
	ADMM & 4096 & 713 & 9.7108e-09 & 25223 \\
	Newton & 4096 & 76 & 8.8869e-07 & 7129 \\ \hline
	\end{tabular}
	\end{table}

\subsubsection{Analysis of numerical results}
In summary, by focusing on the results for Examples 1-3, we can see that ADMM is an effective method for solving CARE. In Example 1, we solved a nine-order continuous-time state space model, and we selected three parameters and found that when the penalty parameters are $\alpha=0.0465,\beta=63.51,\gamma=0.0428$, the iterative steps and CPU time arrived the least, at this time, our algorithm had the best effect. In Example 2, we select CARE with a tridiagonal coefficient matrix. First, we calculate the order of the matrix from 16 to 4096, also the penalty parameters has different choices, but we find that when the penalty parameters are $\alpha =0.91 , \beta= 2.8 , \gamma= 0.0014$, the iterative steps and CPU time arrived the least, at the moment our algorithm has the best effect. However, compared with Newton method, it can be found that the effectiveness of ADMM for solving CARE needs to be improved, and the effectiveness of the algorithm has a great relationship with the selection of penalty parameters, and the selection method of penalty parameters needs to be further studied. Therefore, in order to improve the effectiveness of the algorithm, in the last part of this paper, we decided to combine ADMM and Newton method to solve CARE, which can have better computational effect than using ADMM method alone.

\section{Newton-ADMM}
The numerical solution of nonlinear matrix equations has extensive applications in Physics, Engineering, control theory and other fields. Studying the numerical solution methods of these equations is not only conducive to the development of the equation theory but also great significance in practical applications. In this section, we focus on the numerical solution method of CARE in the control theory system, and the optimization method has always been an important role in Machine Learning, data processing and other fields. As we all know, Newton method is a classical method for solving CARE. Therefore, in this section, we consider transforming CARE into solving a Lyapunov matrix equation by using Newton method, and then we transform the Lyapunov matrix equation into an optimization problem, then use ADMM to solve the Lyapunov matrix equation. The algorithm is expressed in Newton-ADMM. In addition, we give the convergence and numerical results of Newton-ADMM. The experimental results show that the Newton-ADMM can solve CARE effectively.


\subsection{Iterative Method}
Now we introduce the usual scheme of the Newton method for solving the CARE \cite{ref22}.

Define the mapping $\mathcal{R}:\mathbb{R}^{n \times n} \to \mathbb{R}^{n \times n}$ :
\begin{equation*}
\mathcal{F}(X) = A^*X + XA - XNX + K , X \in \mathbb{R}^{n \times n},
\end{equation*}
where $A,N,K$ are define in \cite{ref22}. The first Fréchet detivative of $\mathcal{F}$ at a matrix $X$ is a linear map $\mathcal{F}^{'}_{X}:\mathbb{R}^{n \times n} \to \mathbb{R}^{n \times n}$ given by
\begin{equation*}
\mathcal{F}^{'}_{X}(E) = (A-NX)^TE+E(A-NX).
\end{equation*}
The Newton method for the CARE \cite{ref22} is
\begin{equation}\label{0.1}
X_{k+1} = X_k - (\mathcal{F}^{'}_{X_k})^{-1}\mathcal{F}(X_k) , k = 0,1,...,
\end{equation}
given that the map $\mathcal{F}^{'}_{X_k}$ is invertible. The Newton iteration scheme (\ref{0.1}) is equivalent to
\begin{equation}\label{0.2}
(A-NX)^TX_{k+1} + X_{k+1}(A-NX_k) +X_kNX_k +K = 0.
\end{equation}

Denote $A_k=A-NX_k$ and $Q_k=Q(X_k)=X_kNX_k+K$, where the mapping $Q:\mathbb{R}^{n \times n} \to \mathbb{R}^{n \times n}$. Actually, require $A-NX_k$ is stable, i.e., $\lambda_i(A_k)<0$. The iteration scheme can be written as
\begin{equation}\label{0.3}
A_k^TX_{k+1}+X_{k+1}A_k+Q_k=0.
\end{equation}
It is required to solve the Lyapunov equation (\ref{0.3}) in each Newton iteration step. We recall aan existence and uniqueness theorem for this equation.
\begin{theorem}
Let $A\in \mathbb{R}^{n \times n}$, then the Lyapunov equation
\begin{equation*}
XA+A^T=C,
\end{equation*}
has a uniquely solution $X\in \mathbb{R}^{n \times n}$ for any given matrix $C\in \mathbb{R}^{n \times n}$ if and only if the equation has a uniquely symmetric positive definite solution $X\in \mathbb{R}^{n \times n}$ for any given symmetric positive definite matrix $C\in \mathbb{R}^{n \times n}$.
\end{theorem}
The purpose of this section is to solve CARE by combining ADMM and Newton method. Therefore, we consider transforming the Lyapunov matrix equation into an optimization problem and then solve it in each iteration of Newton method by using ADMM method, and give the iterative algorithm of Newton-ADMM. First, we consider such a Lyapunov matrix equation
\begin{equation}\label{eq3.2}
A^TX+XA+Q = 0,
\end{equation}
where $A,Q\in \mathbb{R}^{n\times n}$ are given matrices, and $Q$ is a symmetric positive definite matrix, $X\in \mathbb{R}^{n\times n}$ is the symmetric solution to be solved.

Now we consider an equivalent form of equation (\ref{eq3.2}):
\begin{equation}\label{eq3.3}
\begin{aligned}
\min_{(X,Y,Z)}  \frac{1}{2} & \left \| Y+ZA+Q \right \|_F^2   \\
s.t. \; A^TX &=Y , \\
X&=Z ,
\end{aligned}
\end{equation}
where $X,Y,Z \in \mathbb{R}^{n \times n}$.

The augmented Lagrangian function of (\ref{eq3.3}) is
\begin{equation}\label{eq3.4}
\begin{split}
\mathcal{L}_{\mathcal{A}}(X,Y,Z,\Lambda,\Pi)&=\frac{1}{2} \left \| Y+ZA+Q \right \|_F^2 -\left \langle \Lambda,A^TX-Y \right \rangle - \left \langle \Pi,X-Z \right \rangle  \\
&+ \frac{\alpha}{2} \left \| A^TX-Y \right \|_F^2 + \frac{\beta}{2} \left \| X-Z \right \|_F^2 ,
\end{split}
\end{equation}
where $\Lambda,\Pi \in \mathbb{R}^{n \times n} $ are Lagrangian multipliers and $\alpha,\beta>0$ are penalty parameters. The ADMM iterations for (\ref{eq3.4}) are given by
\begin{equation*}
\begin{cases}
X_{k+1}:=\mathop{\arg\min}\limits_{X\in \mathbb{R}^{n \times n}} \mathcal{L}_{\mathcal{A}}(X,Y_k,Z_k,\Lambda_k,\Pi_k),\\
Y_{k+1}:=\mathop{\arg\min}\limits_{Y\in \mathbb{R}^{n \times n}} \mathcal{L}_{\mathcal{A}}(X_{k+1},Y,Z_k,\Lambda_k,\Pi_k),\\
Z_{k+1}:=\mathop{\arg\min}\limits_{Z\in \mathbb{R}^{n \times n}} \mathcal{L}_{\mathcal{A}}(X_{k+1},Y_{k+1},Z,\Lambda_k,\Pi_k),\\
\Lambda_{k+1}:=\Lambda_k-\alpha(A^TX_{k+1}-Y_{k+1}),\\
\Pi_{k+1}:=\Pi_k-\beta(X_{k+1}-Z_{k+1}).
\end{cases}
\end{equation*}
Thus, these steps can be written in closed form as
\begin{equation}\label{eq3.5}
\begin{cases}
X_{k+1}=[\alpha AA^T+\beta I_n]^{-1}[A\Lambda_k+\Pi_k+\alpha AY_k+\beta Z_k],\\
Y_{k+1}=(1+\alpha)^{-1}[\alpha A^TX_{k+1}-Z_kA-Q],\\
Z_{k+1}=[(-Y_{k+1}-Q)A^T-\Pi_k+\beta X_{k+1}][AA^T+\beta I_n]^{-1},\\
\Lambda_{k+1}=\Lambda_k-\alpha(A^TX_{k+1}-Y_{k+1}),\\
\Pi_{k+1}=\Pi_k-\beta(X_{k+1}-Z_{k+1}).
\end{cases}
\end{equation}

A point $(X^*,Y^*,Z^*)$ satisfies the KKT optimality conditions for the (\ref{eq3.3}) if there exist $\Lambda^*,\Pi^*,\Gamma^*$ such that
\begin{equation}\label{eq3.6}
\begin{cases}
A\Lambda^*+\Pi^*=0,\\
Y^*+Z^*A+Q+\Lambda^*=0,\\
Z^*AA^T+(Y^*+Q)A^T+\Pi^*=0,\\
A^TX^*-Y^*=0,\\
X^*-Z^*=0.
\end{cases}
\end{equation}

Therefore, combining Newton method with ADMM algorithm, we can get the following Newton-ADMM algorithm for solving CARE.

\renewcommand{\floatpagefraction}{.9}

\begin{algorithm}
\caption{Newton-ADMM}
\renewcommand{\algorithmicrequire}{\textbf{Input:}}
\renewcommand{\algorithmicensure}{\textbf{Output:}}
\begin{algorithmic}\label{alg6}
\REQUIRE Set $A,N,K \in \mathbb{R}^{n\times n}$ and $\epsilon$

%
%

\WHILE{$NRes(X_k)\geqslant \epsilon,k<k_{max}$}

\STATE  1.Calculate $A_k=A-NX_k;Q_k=X_kNX_k+K$,

\STATE  2.Solve $A_k^TX_{k+1}+X_{k+1}A_k+Q_k=0$ from (\ref{eq3.5}) result in $X_{k+1}$,

\STATE  3.Calculate $NRes(X_{k+1})$; Set $k=k+1$.

\ENDWHILE

\end{algorithmic}
\end{algorithm}

We analyze the computational complexity in each iteration of Algorithm \ref{alg6}, it's mainly controlled by matrix multiplication and matrix inverse operation from inner and outer iteration. The calculation of outer iteration is mainly controlled by the iteration format of $A_k$ and $Q_k$, so each outer iteration needs $6n^3-n^2$ computational complexity. The calculation of inner iteration is mainly controlled by the iteration format (\ref{eq3.4}), there are two inverse operations in each iteration, so we use LU decomposition to reduce the computational complexity to $6n^3-n^2$. The iterative formula of $X$ needs $8n^3+4n^2$ computational complexity, the iterative formula of $Y$ needs $4n^3+n^2$ computational complexity, the iterative formula of $Z$ needs $6n^3+3n^2$ computational complexity, the iterative formula of $\Lambda$ needs $3n^2$ computational complexity.
Therefore, the total computational complexity is $39n^3+24n^2$ in inner iteration.

\subsection{Convergent Analysis}

In this section, we will provide the computational complexity of Algorithm \ref{alg6} and establish its convergence. Firstly, we will prove the convergence of ADMM for solving the Lyapunov equation. Subsequently, we will demonstrate the convergence of Newton-ADMM for solving CARE. It is well-known that in general, there is no global convergence guarantee for nonconvex programs or convex programs with three or more blocks. Notably, our iterative formula consists of three blocks. Given these challenges, we present a convergence property of the proposed ADMM that holds under certain assumptions.

Next, we first prove the convergence of ADMM for solving Lyapunov matrix equation.
\begin{theorem}\label{theorem2}
Let ${(X_k,Y_k,Z_k,\Lambda_k,\Pi_k)}$ be a sequence generated by ADMM(\ref{eq3.5}). Suppose the multiplier sequence $(\Lambda_k,\Pi_k)$ is bounded and satisfies
\begin{equation}\label{eq3.3.5}
\sum_{k=0}^\infty (\left \| \Lambda_{k+1}-\Lambda_k \right \|_F^2 + \left \| \Pi_{k+1}-\Pi_k \right \|_F^2 )<\infty.
\end{equation}
Then any accumulation point of ${(X_k,Y_k,Z_k,\Lambda_k,\Pi_k)}$ satisfies the KKT optimality conditions (\ref{eq3.6}) of the problem (\ref{eq3.3}).
\end{theorem}
\begin{proof}
First, we claim
\begin{equation*}
\begin{aligned}
\begin{split}
& \mathcal{L}_{\mathcal{A}}(X_k,Y_k,Z_k,\Lambda_k,\Pi_k)-\mathcal{L}_{\mathcal{A}}(X_{k+1},Y_{k+1},Z_{k+1},\Lambda_{k+1},\Pi_{k+1}) \\
& \geqslant \frac{\beta}{2}\left \| X_k - X_{k+1} \right \|_F^2 + \frac{\alpha}{2}\left \| Y_k - Y_{k+1} \right \|_F^2 + \frac{\beta}{2}\left \| Z_k - Z_{k+1} \right \|_F^2 \\
&-\frac{1}{\alpha}\left \| \Lambda_k - \Lambda_{k+1} \right \|_F^2 -\frac{1}{\beta}\left \| \Pi_k - \Pi_{k+1} \right \|_F^2.
\end{split}
\end{aligned}
\end{equation*}
Since the augmented Lagrangian function of (\ref{eq3.3}) can be rewritten as
\begin{equation}\label{eq3.3.6}
\begin{aligned}
\begin{split}
\mathcal{L}_{\mathcal{A}}(X,Y,Z,\Lambda,\Pi)
&=\frac{1}{2} \left \| Y+ZA+Q \right \|_F^2 +\frac{\alpha}{2} \left \|(A^TX-Y)- \frac{\Lambda}{\alpha} \right \|_F^2 \\
&- \frac{1}{2\alpha} \left \| \Lambda \right \|_F^2 + \frac{\beta}{2} \left \|(X-Z) - \frac{\Pi}{\beta} \right \|_F^2 - \frac{1}{2\beta}\left \| \Pi \right \|_F^2 .
\end{split}
\end{aligned}
\end{equation}
From (\ref{eq3.3.6}), Lemmas \ref{lem11} and \ref{lem13}, obviously, $\mathcal{L}_{\mathcal{A}}$ is strongly convex with respect to each variable of $X$, $Y$ and $Z$, respectively. For $X$, by Lemmas \ref{lem11} and \ref{lem13} as well as the identity matrix has full column rank, then
\begin{equation}\label{eq3.3.8}
\mathcal{L}_{\mathcal{A}}(X+\triangle X,Y,Z,\Lambda,\Pi)-\mathcal{L}_{\mathcal{A}}(X,Y,Z,\Lambda,\Pi) \geqslant \left \langle \partial_X \mathcal{L}_{\mathcal{A}}(X,Y,Z,\Lambda,\Pi),\triangle X \right \rangle + \frac{\beta}{2}\left \| \triangle X \right \|_F^2,
\end{equation}
for any $ X$ and $\triangle X$ and $X$ being a minimizer of $\mathcal{L}_{\mathcal{A}}(X,Y,Z,W,\Lambda,\Pi,\Gamma)$ that we have
\begin{equation}\label{eq3.3.9}
 \left \langle \partial_X \mathcal{L}_{\mathcal{A}}(X,Y,Z,\Lambda,\Pi),\triangle X \right \rangle \geqslant 0.
\end{equation}
Combining (\ref{eq3.3.8}), (\ref{eq3.3.9}) and $X_{k+1}:=\mathop{\arg\min}\limits_{X \in \mathbb{R}^{n \times n}} \mathcal{L}_{\mathcal{A}}(X,Y_k,Z_k,\Lambda_k,\Pi_k)$, we can get
\begin{equation}\label{eq3.3.10}
\mathcal{L}_{\mathcal{A}}(X_k,Y_k,Z_k,\Lambda_k,\Pi_k)-\mathcal{L}_{\mathcal{A}}(X_{k+1},Y_k,Z_k,\Lambda_k,\Pi_k) \geqslant \frac{\beta}{2}\left \| X_k - X_{k+1} \right \|_F^2.
\end{equation}
Similarly, we have
\begin{equation}\label{eq3.3.11}
\mathcal{L}_{\mathcal{A}}(X_{k+1},Y_k,Z_k,\Lambda_k,\Pi_k)-\mathcal{L}_{\mathcal{A}}(X_{k+1},Y_{k+1},Z_k,\Lambda_k,\Pi_k) \geqslant \frac{\alpha}{2}\left \| Y_k - Y_{k+1} \right \|_F^2,
\end{equation}
\begin{equation}\label{eq3.3.12}
\mathcal{L}_{\mathcal{A}}(X_{k+1},Y_{k+1},Z_k,\Lambda_k,\Pi_k)-\mathcal{L}_{\mathcal{A}}(X_{k+1},Y_{k+1},Z_{k+1},\Lambda_k,\Pi_k) \geqslant \frac{\beta}{2}\left \| Z_k - Z_{k+1} \right \|_F^2.
\end{equation}
Due to $\Lambda_{k+1} = \Lambda_k-\alpha(A^TX_{k+1}-Y_{k+1})$, then
\begin{equation}\label{eq3.3.14}
\begin{aligned}
\begin{split}
&\mathcal{L}_{\mathcal{A}}(X_{k+1},Y_{k+1},Z_{k+1},\Lambda_k,\Pi_k)-\mathcal{L}_{\mathcal{A}}(X_{k+1},Y_{k+1},Z_{k+1},\Lambda_{k+1},\Pi_k)\\
&= -\left \langle  \Lambda_k,A^TX_{k+1}-Y_{k+1} \right \rangle + \left \langle  \Lambda_{k+1},A^TX_{k+1}-Y_{k+1} \right \rangle\\
& = \left \langle \Lambda_{k+1}-\Lambda_k,\frac{\Lambda_k-\Lambda_{k+1}}{\alpha} \right \rangle\\
&= -\frac{1}{\alpha} \left \| \Lambda_k - \Lambda_{k+1} \right \|_F^2.
\end{split}
\end{aligned}
\end{equation}
We have similarly
\begin{equation}\label{eq3.3.15}
\mathcal{L}_{\mathcal{A}}(X_{k+1},Y_{k+1},Z_{k+1},\Lambda_{k+1},\Pi_k)-\mathcal{L}_{\mathcal{A}}(X_{k+1},Y_{k+1},Z_{k+1},\Lambda_{k+1},\Pi_{k+1}) = -\frac{1}{\beta} \left \| \Pi_k - \Pi_{k+1} \right \|_F^2.
\end{equation}
Taking summation of (\ref{eq3.3.10})-(\ref{eq3.3.15}), we have
\begin{equation}\label{eq3.3.17}
\begin{aligned}
\begin{split}
& \mathcal{L}_{\mathcal{A}}(X_k,Y_k,Z_k,\Lambda_k,\Pi_k)-\mathcal{L}_{\mathcal{A}}(X_{k+1},Y_{k+1},Z_{k+1},\Lambda_{k+1},\Pi_{k+1}) \\
& \geqslant \frac{\beta}{2}\left \| X_k - X_{k+1} \right \|_F^2 + \frac{\alpha}{2}\left \| Y_k - Y_{k+1} \right \|_F^2 + \frac{\beta}{2}\left \| Z_k - Z_{k+1} \right \|_F^2\\
&-\frac{1}{\alpha}\left \| \Lambda_k - \Lambda_{k+1} \right \|_F^2 -\frac{1}{\beta}\left \| \Pi_k - \Pi_{k+1} \right \|_F^2 .
\end{split}
\end{aligned}
\end{equation}

Now, we show $(X_{k+1},Y_{k+1},Z_{k+1},\Lambda_{k+1},\Pi_{k+1})-(X_k,Y_k,Z_k,\Lambda_k,\Pi_k) \to 0$, due to (\ref{eq3.3.5}) so that $(\Lambda_{k+1},\Pi_{k+1})-(\Lambda_k,\Pi_k) \to 0$, so we only need to prove that $(X_{k+1},Y_{k+1},Z_{k+1})-(X_k,Y_k,Z_k)\to 0$, from the boundedness of ${(\Lambda_k,\Pi_k}$ and (\ref{eq3.3.6}), we obtain that $\mathcal{L}_{\mathcal{A}}(X,Y,Z,\Lambda_k,\Pi_k)$ is bounded below for any $X,Y,Z,k$. Taking summation of the above inequality (\ref{eq3.3.17}) and note that $\mathcal{L}_{\mathcal{A}}(X,Y,Z,\Lambda_k,\Pi_k)$ is bounded below, we get
\begin{equation*}
\begin{split}
&\sum_{k=0}^\infty (\frac{\beta}{2}\left \| X_k - X_{k+1} \right \|_F^2 + \frac{\alpha}{2}\left \| Y_k - Y_{k+1} \right \|_F^2 + \frac{\beta}{2}\left \| Z_k - Z_{k+1} \right \|_F^2\\
&- \sum_{k=0}^\infty (\frac{1}{\alpha}\left \| \Lambda_k - \Lambda_{k+1} \right \|_F^2 +\frac{1}{\beta}\left \| \Pi_k - \Pi_{k+1} \right \|_F^2 ) < \infty,
\end{split}
\end{equation*}
since $\sum_{k=0}^\infty (\left \| \Lambda_k - \Lambda_{k+1} \right \|_F^2 +\left \| \Pi_k - \Pi_{k+1} \right \|_F^2  ) < \infty$, it holds
\begin{equation*}
\sum_{k=0}^\infty (\frac{\beta}{2}\left \| X_k - X_{k+1} \right \|_F^2 + \frac{\alpha}{2}\left \| Y_k - Y_{k+1} \right \|_F^2 + \frac{\beta}{2}\left \| Z_k - Z_{k+1} \right \|_F^2 ) < \infty,
\end{equation*}
thus
\begin{equation*}
(X_{k+1},Y_{k+1},Z_{k+1})-(X_k,Y_k,Z_k)\to 0.
\end{equation*}
Finally, we are ready to prove the results of this theorem. From (\ref{eq3.5}), we have
\begin{equation}\label{eq3.3.18}
\begin{aligned}
[\alpha AA^T+\beta I_n](X_{k+1}-X_k)&=A\Lambda_k+\Pi_k+\alpha AY_k+\beta Z_k-\alpha AA^TX_k-\beta X_k\\
&=A\Lambda_k+\Pi_k+\alpha A(Y_k-A^TX_k)+\beta(Z_k-X_k),
\end{aligned}
\end{equation}
\begin{equation}\label{eq3.3.19}
\begin{aligned}
(1+\alpha)(Y_{k+1}-Y_k)&=\alpha A^TX_{k+1}-Z_kA-Q-\Lambda_k-Y_k-\alpha Y_k\\
&=\alpha (A^TX_{k+1}-Y_{k+1})-Z_kA-Q-\Lambda_k-Y_k+\alpha(Y_{k+1}-Y_k),
\end{aligned}
\end{equation}
\begin{equation}\label{eq3.3.20}
\begin{aligned}
(Z_{k+1}-Z_k)[AA^T+\beta I_n]
&=(-Y_{k+1}-Q)A^T-\Pi_k+\beta X_{k+1}-Z_kAA^T-\beta Z_k\\
&=\beta(X_{k+1}-Z_{k+1})+\beta(Z_{k+1}-Z_k)+ (-Y_{k+1}-Q)A^T-\Pi_k-Z_kAA^T,
\end{aligned}
\end{equation}

\begin{equation}\label{eq3.3.22}
\Lambda_{k+1}-\Lambda_k=-\alpha(A^TX_{k+1}-Y_{k+1}),
\end{equation}

\begin{equation}\label{eq3.3.23}
\Pi_{k+1}-\Pi_k=-\beta(X_{k+1}-Z_{k+1}).
\end{equation}

Since $(\Lambda_k,\Pi_k)$ is bounded, then ${(X_k,Y_k,Z_k)}$ is bounded. And $(X_{k+1},Y_{k+1},Z_{k+1},\Lambda_{k+1},\Pi_{k+1})-(X_k,Y_k,Z_k,\Lambda_k,\Pi_k) \to 0$, let the left and right hand sides in (\ref{eq3.3.18})-(\ref{eq3.3.23}) all go to zero so that
\begin{equation}\label{eq3.3.25}
\begin{aligned}
&A\Lambda_k+\Pi_k \to 0,\\
&\alpha (A^TX_{k+1}-Y_{k+1})-Z_kA-Q-\Lambda_k-Y_k \to 0,\\
&(-Y_{k+1}-Q)A^T-\Pi_k-Z_kAA^T \to 0,\\
&(Y_{k+1}+Z_{k+1}A+K)X_{k+1}^T \to 0,\\
&A^TX_{k+1}-Y_{k+1}\to 0,\\
&X_{k+1}-Z_{k+1}\to 0.
\end{aligned}
\end{equation}

For any limit point $(X^*,Y^*,Z^*)$ of the sequence ${(X_k,Y_k,Z_k)}$, there exists a subsequence ${X_{k_i},Y_{k_i},Z_{k_i}}$ converging to $(X^*,Y^*,Z^*)$. The boundedness of $(\Lambda_k,\Pi_k)$ implies the existence of a sub-subsequence ${(\Lambda_{k_{i_j}},\Pi_{k_{i_j}})}$ of ${(\Lambda_{k_i},\Pi_{k_i})}$ converging to some point $(\Lambda^*,\Pi^*)$. Thus, $(X^*,Y^*,Z^*,\Lambda^*,\Pi^*)$ is a limit point of ${(X_k,Y_k,Z_k,\Lambda_k,\Pi_k)}$. Taking limitation of (\ref{eq3.3.25}), it holds that
\begin{equation*}
\begin{cases}
A\Lambda^*+\Pi^*=0,\\
Y^*+Z^*A-W^*X^*+Q+\Lambda^*=0,\\
Z^*AA^T+(Y^*+K)A^T+\Pi^*=0,\\
A^TX^*-Y^*=0,\\
X^*-Z^*=0,
\end{cases}
\end{equation*}
Hence, we get the proof completely.
\end{proof}

We can get the following corollary by Theorem \ref{theorem2}.
\begin{corollary}
Let ${(X_k,Y_k,Z_k,\Lambda_k,\Pi_k)}$ be a sequence generated by ADMM(\ref{eq3.4}), when the sequence converges, the limit satisfies the KKT optimality conditions (\ref{eq3.3}) of the problem (\ref{eq3.6}).
\end{corollary}
The following theorem \cite{ref22} describes the convergence property of the iterative scheme (\ref{0.2}):
\begin{theorem}
For the CARE (\ref{eq1.1}), if the pair $(A,N)$ is stabilizable and $(K,A)$ is detectable, and $X_0$ is symmetric positive semi-definite so that $A-NX_0$ is stable, the sequence of matrices ${X_k}$ determined by (\ref{0.2}) quadratic converges to the symmetric positive semi-definite solution $X^*$ of (\ref{eq1.1}). In other words, there exists a constant $\delta>0$ so that
\begin{equation*}
\left \| X_k-X^* \right \|_2  \leqslant \delta \left \| X_{k-1}-X^* \right \|_2^2,
\end{equation*}
for all $k$.
\end{theorem}
Therefore, the convergence of Newton-ADMM algorithm is proved completely.

\subsection{Experimental Results}
In this section, we provide two numerical examples to demonstrate the efficiency of Algorithm \ref{alg6}. Additionally, we compare its performance with Newton's method for solving (\ref{eq1.1}), and conclude that Algorithm \ref{alg6} is more efficient in solving CARE. All the code is implemented in Matlab. The iteration step and error of the objective function are denoted as iterations and errors, respectively. We take the matrix order "$n$" as 16, 32, 64, 128, 256, 512, 1024, 2048 and 4096, initialize $X_0$ as an null matrix, and error precision is $\epsilon=10^{-8}$. We use the error of matrix equation as
\begin{equation*}
\left \| R_k \right \|_F=\left \| A^TX+XA-XNX+K \right \|_F.
\end{equation*}

\subsubsection{Example 1}
We compare Newton-ADMM algorithm with Newton algorithm, the given coefficient matrices are
\begin{equation*}
A =
\begin{pmatrix}
6 & 1 & & & & \\
2 & 6 & 1 & &  &\\
 & \ddots & \ddots & \ddots &\\
 & & \ddots & \ddots & 1 &\\
 & & & 2& 6& \\
\end{pmatrix}_{n \times n}, \; \;
B^T =
\begin{pmatrix}
5 & 1 & & & & \\
2 & 5 & 1 & &  &\\
 & \ddots & \ddots & \ddots &\\
 & & \ddots & \ddots & 1 &\\
 & & & 2 & 5& \\
\end{pmatrix}_{n \times n},
\end{equation*}

\begin{equation*}
N=BB^T, K=I_n.
\end{equation*}
When we set the parameter are $\alpha =0.8 , \beta= 53.5 $ that get the best calculation effect. Then when the matrix order is different, the error data is shown in Table 9.
\begin{table}[!htbp]
	\centering
	\label{tab9}
 \caption{Numerical results for Example 1}
	\begin{tabular}{ccccc}
	\hline
	algorithm & n & iteration & error & time(s) \\ \hline
	Newton-ADMM & 16 & 448 & 2.5323e-10 & 0.03 \\
	Newton & 16 & 83 & 6.0989e-08 & 0.09 \\\hline
	Newton-ADMM & 32 & 453 & 3.7771e-10 & 0.08 \\
	Newton & 32 & 83 & 6.9049e-08 & 0.18 \\\hline
	Newton-ADMM & 64 & 455 & 4.0151e-10 & 0.18 \\
	Newton & 64 & 83 & 6.8379e-08 & 0.89 \\\hline
	Newton-ADMM & 128 & 467 & 4.1344e-10 & 1.09 \\
	Newton & 128 & 83 & 7.0817e-08 & 2.27 \\\hline
	Newton-ADMM & 256 & 472 & 4.1654e-10 & 4.32 \\
	Newton & 256 & 83 & 7.7239e-08 & 6.51 \\\hline
	Newton-ADMM & 512 & 474 & 4.1731e-10 & 20 \\
	Newton & 512 & 83 & 8.3165e-08 & 28 \\\hline
	Newton-ADMM & 1024 & 488 & 4.1751e-10 & 103 \\
	Newton & 1024 & 83 & 9.6813e-08 & 148 \\\hline
	Newton-ADMM & 2048 & 491 & 4.1757e-10 & 887 \\
	Newton & 2048 & 83 & 9.7205e-08 & 1037 \\\hline
	Newton-ADMM & 4096 & 493 & 4.1758e-10 & 9657 \\
	Newton & 4096 & 83 & 1.1359e-07 & 11048 \\  \hline

	\end{tabular}	
	\end{table}

\subsubsection{Example 2}
We compare Newton-ADMM algorithm with Newton algorithm similarly, the given coefficient matrices are
\begin{equation*}
A =
\begin{pmatrix}
3 & 1 & & & & \\
1 & 3 & 1 & &  &\\
 & \ddots & \ddots & \ddots &\\
 & & \ddots & \ddots & 1 &\\
 & & & 1& 3& \\
\end{pmatrix}_{n \times n}, \; \;
B^T =
\begin{pmatrix}
6 & 2 & & & & \\
2 & 6 & 2 & &  &\\
 & \ddots & \ddots & \ddots &\\
 & & \ddots & \ddots & 2 &\\
 & & & 2 & 6& \\
\end{pmatrix}_{n \times n},
\end{equation*}

\begin{equation*}
N=BB^T, K=I_n.
\end{equation*}
When we set the parameter are $\alpha =0.8 , \beta= 45 $ that get the best calculation effect. Then when the matrix order is different, the error data is shown in Table 10.
\begin{table}[!htbp]
	\centering
	\label{tab10}
 \caption{Numerical results for Example 2}
	\begin{tabular}{ccccc}
	\hline
	algorithm & n & iteration & error & time(s) \\ \hline
	Newton-ADMM & 16 & 373 & 4.0583e-11 & 0.02 \\
	Newton & 16 & 76 & 6.0918e-08 & 0.05 \\\hline
	Newton-ADMM & 32 & 353 & 5.0978e-11 & 0.06 \\
	Newton & 32 & 76 & 6.2125e-08 & 0.11 \\\hline
	Newton-ADMM & 64 & 361 & 6.1431e-11 & 0.17 \\
	Newton & 64 & 76 & 6.5353e-08 & 0.52 \\\hline
	Newton-ADMM & 128 & 362 & 6.4471e-11 & 0.87 \\
	Newton & 128 & 76 & 6.2467e-08 & 1.38 \\\hline
	Newton-ADMM & 256 & 367 & 6.5266e-11 & 3.88 \\
	Newton & 256 & 76 & 7.0731e-08 & 4.97 \\\hline
	Newton-ADMM & 512 & 374 & 6.5468e-11 & 18 \\
	Newton & 512 & 76 & 7.5188e-08 & 23 \\\hline
	Newton-ADMM & 1024 & 378 & 6.5520e-11 & 88 \\
	Newton & 1024 & 76 & 7.8919e-08 & 104 \\\hline
	Newton-ADMM & 2048 & 383 & 6.5531e-11 & 834 \\
	Newton & 2048 & 76 & 8.3520e-08 & 983 \\\hline
	Newton-ADMM & 4096 & 390 & 6.5537e-11 & 6534 \\
	Newton & 4096 & 76 & 8.8869e-07 & 7129 \\ \hline
	\end{tabular}
	\end{table}
	
\subsubsection{Analysis of numerical results}

In summary, by focusing on the results for Examples 1-2, we can see that Newton-ADMM is an effective method for solving CARE. In Example 1, we solve CARE with tridiagonal coefficient matrices, and we selected three parameters and found that when the penalty parameters are $\alpha=0.8,\beta=53.5$, the iterative steps and CPU time arrived the least. In Example 2, we solve CARE with the symmetric tridiagonal coefficient matrices. First, we calculate the order of the matrix from 16 to 4096, also the penalty parameters has different choices, but we find that when the penalty parameters are $\alpha =0.8 , \beta= 45$, the iterative steps and CPU time arrived the least. Further, compared with Newton method, it can be found that the effectiveness of ADMM for solving CARE.

\section{Conclusion}
The objective of this paper is to solve matrix equations using various optimization algorithms. Firstly, we transform the matrix equation into an optimization problem and apply classical optimization algorithms to solve it. We provide the computational complexity and convergence analysis of the algorithm in this paper. Furthermore, numerical examples are presented to demonstrate the effectiveness of these methods. When compared with other classical methods, our methods exhibit both advantages and disadvantages, with the latter requiring further improvement.






\end{document}